\documentclass[12pt]{article}
\usepackage[margin=1in]{geometry}
\usepackage{amsmath,amssymb,amsthm,array,colortbl,rotating,graphicx}
\usepackage{color}
\usepackage{natbib}
\usepackage[colorlinks]{hyperref}
\usepackage{multicol,multirow,enumitem,booktabs,mathtools}
\usepackage{microtype}
\usepackage{bm}
\usepackage{tikz}
\usetikzlibrary{arrows,fit}

\def\highlightchangessince{62}

\definecolor{changecolor}{RGB}{0,170,255}
\def\bct#1#2\ect#3{\ifnum#1>\highlightchangessince\textcolor{changecolor}{\ignorespaces#2\unskip}\else\ignorespaces#2\unskip\fi}
\newcounter{changeversion}\newcounter{outerchangeversion}
\def\bc#1{\ifnum\thechangeversion>0\setcounter{outerchangeversion}{\thechangeversion}\setcounter{changeversion}{#1}\else\setcounter{changeversion}{#1}\fi\ifnum#1>\highlightchangessince\color{changecolor}\fi\ignorespaces}
\def\ec#1{\unskip\ifnum\theouterchangeversion>0\setcounter{changeversion}{\theouterchangeversion}\setcounter{outerchangeversion}{0}\else\setcounter{changeversion}{0}\fi\ifnum\thechangeversion>\highlightchangessince\else\normalcolor\fi}

\let\originalleft\left
\let\originalright\right
\renewcommand{\left}{\mathopen{}\mathclose\bgroup\originalleft}
\renewcommand{\right}{\aftergroup\egroup\originalright}

\def\bas#1\eas{\begin{align*}#1\end{align*}}
\def\basn#1\easn{\begin{align}#1\end{align}}

\newtheorem{thm}{Theorem}
\newtheorem{cor}[thm]{Corollary}
\newtheorem{lem}[thm]{Lemma}
\theoremstyle{definition}
\newtheorem{defn}[thm]{Definition}
\newtheorem{ex}[thm]{Example}
\newtheorem{assum}[thm]{Assumption}
\newtheorem{rem}[thm]{Remark}

\numberwithin{thm}{section}
\numberwithin{equation}{section}

\newlength\minuswidth
\newcommand{\wideplus}{\makebox[\minuswidth]{${}+{}$}}

\makeatletter\renewcommand{\paragraph}{\@startsection{paragraph}{4}%
  {\z@}{1.5ex \@plus .2ex}{-1em}%
  {\normalfont\normalsize\bfseries}}\makeatother

\newcommand{\SADS}{\SwapAboveDisplaySkip}

\newcommand{\T}{\mathrm{\scriptscriptstyle T}}
\renewcommand{\epsilon}{\varepsilon}
\newcommand{\given}{\mathbin\vert}
\newcommand{\Given}{\:\middle\vert\:}

\DeclareMathOperator{\Var}{Var}

\begin{document}

\title{%
From Random Walks to Random Leaps: Generalizing Classic Markov Chains for Big Data Applications
}

\date{\today}
\author{Bala Rajaratnam, Narut Sereewattanawoot,\\Doug Sparks, and Meng-Hsuan Wu}

\maketitle

\begin{abstract}\noindent
Simple random walks are a basic staple of the foundation of probability
theory and form
the
building block of many useful and complex stochastic
processes. In this paper we
study
a natural
generalization of the random walk to a process in which the allowed step
sizes take values in the set $\{\pm1,\pm2,\ldots,\pm k\}$, a process we
call a \emph{random leap}. The need to analyze such models arises naturally
in modern-day
data science and so-called ``big data''
applications. We provide closed-form expressions for
quantities associated with first passage times and absorption events of
random leaps.
These expressions are formulated in terms of the roots of the characteristic polynomial of a certain recurrence relation associated with the transition probabilities.
Our analysis shows that the expressions for absorption
probabilities for the classical simple random walk are a special case of a
universal
result that is very elegant.
We also consider an important variant of a random
leap: the reflecting random leap. We demonstrate that the reflecting
random leap exhibits more interesting behavior in regard to the
existence of a stationary distribution and properties thereof. Questions
relating to recurrence/transience are also addressed, as well as an
application of the random leap.



\end{abstract}

\section{Introduction}
\label{sec:intro}

The random walk is among the most ubiquitous and thoroughly studied of all stochastic processes.  Its analytically tractable structure and extraordinary applicability in fields as diverse as
genetics, the environmental and geo-sciences,
economics,
computer science, and psychology have made it an appealing target of research.
Moreover, its accessibility and the existence of closed-form solutions
for important quantities of interest
have made it a customary example in courses and texts on stochastic processes.
It is therefore natural to ask how the random walk can be extended, and to ask how far such extensions can be pursued while still retaining tractability.  One immediately logical extension is to allow steps of
sizes in the set $\{\pm1,\pm2,\ldots,\pm k\}$ for some integer~$k$, which we call the 
\emph{random leap}.

Our motivation for considering random leaps is
fourfold.
First, there exists the perception that the random walk itself is the only model of its type for which a variety of quantities and results are available in closed form.
We show that the random walk is actually a special case of a larger class of stochastic processes in which tractable solutions exist for a collection of problems.  Second, in modern
high-dimensional statistics, computer science, machine learning,
data science,
and allied fields,
an important goal is to
parsimoniously explain observed data with a probabilistic model that uses only a subset of all available
features/covariates.
As the number of possible
candidate models that can ``fit'' the data
may be extremely large (due to the sheer size of modern ``big data'' applications), a model is often selected through a stochastic exploration of model
space.
If such a process is permitted to add or remove only one
variable/feature
at a time, then it may be conceptualized as a random walk on model space, where models that differ only through the inclusion or exclusion of a single variable are considered adjacent.  However, the model space may be searched much more efficiently if the process can add or remove several variables at once, i.e., if it is permitted to move directly to non-adjacent models \citep[see, e.g.,][and the references therein]{hans2007}.  Such a process is a random leap on model space.
Analysis of the one-dimensional random leap on the integers is thus a natural first step toward understanding the behavior of these more flexible procedures for exploring very large model spaces.
%
%
Third, our results and their method of proof exhibit 
important
connections to other branches of mathematics, particularly algebra, through their use of concepts such as recurrence relations, characteristic polynomials, and determinants.
Fourth and finally, there
is also some degree of elegance to be found in the demonstration that familiar results
for the classical random walk
are in fact special cases of
a more
universal
phenomenon.

The
outline of this paper is as follows.
In Section~\ref{sec:srw}, we first compare a standard result for random walks to corresponding results for the two-step case of the random leap, which we call the \emph{random jump}.  This section both provides intuition and demonstrates the beginnings of a pattern. 
Section~\ref{sec:leap} begins by defining the random leap in general and discusses its connections to other models.  We then derive closed-form expressions for both absorption probabilities and expected absorption times for a random leap with absorbing barriers.
Section~\ref{sec:rtsd} considers the related issues of recurrence/transience and stationary distributions for random leaps on~$\mathbb{Z}$.  Section~\ref{sec:reflect} examines an important variant
termed
the \emph{reflecting random leap}.  This section establishes results on the existence and form of a stationary distribution for models with both a finite and an infinite state space.  Section~\ref{sec:example} gives a concrete application of the random jump theory to the casino game of roulette, and Section~\ref{sec:concl} states some concluding remarks.
Table~\ref{tab:summary} below summarizes our main theoretical results and their locations in the paper.

\begin{table}[htbp]
\centering
\begin{tabular}{lccc}\toprule
&\multicolumn{3}{c}{Random Leap Type}\\\cmidrule(l){2-4}
&&\multicolumn{2}{c}{Reflecting}\\\cmidrule(l){3-4}
\multicolumn{1}{c}{Result Type}&Simple&One-Sided&Two-Sided\\\midrule
Absorption Probabilities&Theorem~\ref{thm:u}&n/a&n/a\\
Expected Absorption Times&Theorem~\ref{thm:v}&n/a&n/a\\
Stationary Distribution&Theorem~\ref{thm:rtsd}\phantom0&Theorem~\ref{thm:pi-limit}&Theorem~\ref{thm:pi}\\\bottomrule
\end{tabular}
\caption{Location of main theoretical results and their location within the paper.  Note that some results for some types of random leap are trivial or not applicable~(n/a).  Also note that although Theorem~\ref{thm:rtsd} proves the nonexistence of a stationary distribution for the simple random leap from first principles, the result is a special case of a more general result of \citet{chung1951}.}
\label{tab:summary}
\end{table}

\section{Random Walks and Random Jumps}
\label{sec:srw}

Our ultimate goal is to elucidate the behavior of a full extension of the simple random walk to an arbitrary maximum step size~$k$ in each direction.  However, some insight may first be gained by considering the simple random walk
itself and
its extension to the simple case where $k=2$.

\subsection{Established Results for Random Walks}
\label{subsec:walk}

We begin by summarizing some standard results on random walks to provide context against which to compare our subsequent findings for random leaps.  We state these results without proof since they may be found in typical introductory texts on probability and stochastic processes \citep[e.g.,][]{feller1968,karlin1998} and since they will moreover follow as special cases of more general results to be presented later.

First, for any discrete-time Markov chain $\{X_n:n\ge0\}$ with state space $\{0,\ldots,N\}$ for which $0$ and $N$ are absorbing barriers, let $u_i=P(X_T=N\given X_0=i)$ be the upper absorption probability from state~$i$,
where $T=\min\{n\ge0:X_n=0\text{ or }X_n=N\}$.  
Now specifically take $\{X_n:n\ge0\}$ to be the absorbing random walk on $\{0,\ldots,N\}$ that takes rightward steps with probability $p$ and leftward steps with probability $1-p$.

Note that the construction described above can be alternatively interpreted in terms of first passage times for a simple (non-absorbing) random walk on~$\mathbb{Z}$.  Let $\{X_n^\star:n\ge0\}$ be such a simple random walk.  Now define the sets $S_0=\{m\in\mathbb{Z}:m\le0\}$ and
$S_{
N
}=\{m\in\mathbb{Z}:m\ge N\}$,
and let $T^\star_{S_0}=\min\{n\ge0:X_n^\star\in S_0\}$ and $T^\star_{S_N}=\min\{n\ge0:X_n^\star\in S_N\}$ denote the first passage times for these sets.  Then $u_i=P(T^\star_{S_N}<T^\star_{S_0})$.

Regardless of how the quantities $u_0,u_1,\ldots,u_N$ are interpreted, their form for the random walk is provided by the following lemma.

\begin{lem}[e.g., \citeauthor{feller1968}, \citeyear{feller1968}; \citeauthor{karlin1998}, \citeyear{karlin1998}]\label{lem:u-1}
For the absorbing simple random walk
with $0<p<1$,
the upper absorption probabilities
are
\[
u_i=\frac{\sum_{j=1}^iz_1^j}{\sum_{j=1}^Nz_1^j},
\]
where $z_1=(1-p)/p$.
\end{lem}

It is clear that the
upper absorption probabilities~$u_i$
form a sequence that increases from $0$ to $1$ as $i$ increases from $0$ to~$N$, but the form of the quantity $z_1$ is also of interest.  It can be shown that the differences $d_i=u_i-u_{i-1}$ are related by the recurrence relation $d_i=[(1-p)/p]d_{i-1}$ for $i\ge2$.
The characteristic polynomial for this recurrence relation is $\chi^\star(z)=z-(1-p)/p$, which has root $(1-p)/p=z_1$.  Thus, $u_i$ may be alternatively formulated as
\basn\SADS\label{u-det-srw}
u_i=\frac{\det(\bm{A}_{\mathrlap{\,i}\hphantom{N}}\bm{Z})}{\det(\bm{A}_N\bm{Z})},
\easn
where $\bm{A}_i=(\bm{1}_i^\T\;\;\bm{0}_{N-i}^\T)$ is a $1\times N$
row vector
and $\bm{Z}=(z_1\;\;z_1^2\;\;\cdots\;\;z_1^N)^\T$ is an $N\times1$
column vector,
noting that the determinant coincides with the identity function for scalar-valued arguments.  The motivation for
the
alternative expression
in~(\ref{u-det-srw})
will become clear in the next subsection.

\subsection{Random Jumps (Two Steps)}
\label{subsec:jump}

It is natural to ask how the results for the random walk might extend to a Markov chain that can take steps of size $1$ or $2$ in each direction.  We begin by defining terminology
for
such stochastic processes.

\begin{defn}
Let
$p_1,p_2,q_1,q_2$
be nonnegative real numbers such that
$p_1+p_2+q_1+q_2=1$.
Now consider an integer-valued Markov chain $\{X_n:n\ge0\}$ with transition probabilities for every $n\ge0$ given by
\bas\SADS
P\left(X_{n+1}=j\Given X_n=i\right)=\begin{cases}
p_{j-i}&\text{ if }i\wideplus1\le j\le i\wideplus2,\\
q_{i-j}&\text{ if }i-2\le j\le i-1,\\
0&\text{ if }|i-j|>2\text{ or }i=j.
\end{cases}
\eas
We call this Markov chain a \emph{simple random jump}.
\end{defn}

Note that we use the term ``jump'' since the chain is permitted to to ``jump'' over an adjacent state to reach a state two units away in a single time step.
We now proceed to define an absorbing random jump.

\begin{defn}
Let $\{X_n^\star:n\ge0\}$ be a simple random jump, and let $N\ge4$ be an integer.  Now let $T^\star=\min\{n\ge0:X_n^\star\le0\text{ or }X_n^\star\ge N\}$, and define a Markov chain $\{X_n:n\ge0\}$ by
\begin{align*}
X_n=\begin{cases}
X_n^\star &\text{ if }n<T^\star,\\
0 &\text{ if }n\ge T\text{ and }X^\star_{T^\star}\le0,\\
N &\text{ if }n\ge T\text{ and }X^\star_{T^\star}\le N.
\end{cases}
\end{align*}
We call the
above Markov chain~$X_n$
an \emph{absorbing random jump}.
\end{defn}

Note that the absorbing random jump is absorbed when it would otherwise ``jump past'' the barrier (i.e., when it would otherwise move from state~$1$ to state~$-1$, ``jumping past'' the barrier at~$0$, or when it would otherwise move from state $N-1$ to state $N+1$, ``jumping past'' the barrier at~$N$).

The analysis of the upper absorption probabilities for the random jump parallels that for the random walk.
Suppose $p_2$ and $q_2$ are both nonzero.  Then it
can again be shown that the differences $d_i=u_i-u_{i-1}$ satisfy a recurrence relation for $i\ge3$.  However, this recurrence relation is now of order~$3$ and has characteristic polynomial $\chi^\star(z)=[p_2z^3+(p_1+p_2)z^2-(q_1+q_2)z-q_2]/p_2$.  Let $z_1$, $z_2$, and $z_3$ denote the roots of this polynomial.
\begin{rem}
It can be shown that when $p_2$ and $q_2$ are both nonzero, these roots are distinct and nonzero, though we do not prove this fact here.
\end{rem}
Then the following result can be shown.
\begin{thm}\label{thm:u-2}
For the absorbing random jump
with $p_2$ and $q_2$ both nonzero,
the upper absorption probabilities are
\[
u_i=\frac{
(z_2^{N+1}-z_3^{N+1})\sum_{j=1}^i z_1^j+(z_3^{N+1}-z_1^{N+1})\sum_{j=1}^i z_2^j+(z_1^{N+1}-z_2^{N+1})\sum_{j=1}^i z_3^j
}
{
(z_2^{N+1}-z_3^{N+1})\sum_{j=1}^N z_1^j+(z_3^{N+1}-z_1^{N+1})\sum_{j=1}^N z_2^j+(z_1^{N+1}-z_2^{N+1})\sum_{j=1}^N z_3^j
}.
\]
\end{thm}
It
will be seen in Theorem~\ref{thm:u} that the above result is
compactly
stated as
\basn\label{u-det-srj}
u_i&=\begin{vmatrix}1&1&1\\ \sum_{j=1}^N z_1^j&\sum_{j=1}^N z_2^j&\sum_{j=1}^N z_3^j\\ z_1^{N+1}&z_2^{N+1}&z_3^{N+1}\end{vmatrix}^{-1}\begin{vmatrix}1&1&1\\ \sum_{j=1}^i z_1^j&\sum_{j=1}^i z_2^j&\sum_{j=1}^i z_3^j\\ z_1^{N+1}&z_2^{N+1}&z_3^{N+1}\end{vmatrix}
=\frac{\det(\bm{A}_{\mathrlap{\,i}\hphantom{N}}\bm{Z})}{\det(\bm{A}_N\bm{Z})},
\easn
where $|\,\cdot\,|$ denotes the determinant, $\bm{Z}$ is the $(N+2)\times3$ matrix with $(i,j)$th element $Z_{i,j}=z_j^{i-1}$, and
\begin{align*}\SADS
\bm{A}_i=\begin{pmatrix}1&\bm{0}^\T_i&\bm{0}_{N-i}^\T&0\\0&\bm{1}_i^\T&\bm{0}_{N-i}^\T&0\\0&\bm{0}_i^\T&\bm{0}_{N-i}^\T&1\end{pmatrix},
\end{align*}
where $\bm{1}_i$ denotes a length-$i$ vector of
ones.
Thus, the determinant-based result in~(\ref{u-det-srj}) for the simple random jump above may be seen as a generalization of the analogous determinant-based result in~(\ref{u-det-srw}) for the simple random walk.  The differences are that (i)~the matrix~$\bm Z$ has two additional columns corresponding to the two additional roots of the characteristic polynomial and (ii)~the matrices~$\bm A_i$ have two additional rows, each of which is simply a unit row vector.
We will directly generalize such concepts to larger values of the maximum step size~$k$ in the next section.

\section{Random Leaps: First
Passage
Times and \\Absorption Events}
\label{sec:leap}

We now extend the results for random jumps to Markov chains that can take up to $k$~steps in either direction, where $k$ can be any integer.
We call such Markov chains \emph{random leaps}.
After first defining some preliminaries in Subsection~\ref{subsec:leap-pre}, we discuss the connections between absorbing random leaps and other types of models in Subsection~\ref{subsec:leap-graph}.  We then derive upper absorption probabilities in Subsection~\ref{subsec:leap-u} and expected absorption times in Subsection~\ref{subsec:leap-v}.

The basic approach employed throughout this section begins by showing that the sequence of differences of each quantity of interest obey linear recurrence relations.
The roots of the characteristic polynomials of these recurrence relations form a key component of our eventual closed-form solutions, and the resulting expressions suggest connections to algebraic varieties.
Another key feature of our results is that they involve only matrices that are much smaller than the full transition matrix of the Markov chain.  In particular, the matrices that appear in all of our results are of dimension $2k\times2k$ or smaller.

In particular, one of our key results in this section (Theorem~\ref{thm:u}) shows that the determinant-based interpretation of Theorem~\ref{thm:u} in~(\ref{u-det-srj}) directly generalizes to any
arbitrary
maximum step size~$k>0$.
For larger~$k$, the recurrence relation associated with the differences~$d_i$ will be of higher order, and hence there will be more roots and accordingly more columns in the matrix~$\bm{Z}$.  However, these cases will also require more boundary conditions, which will introduce a corresponding additional number of rows at the top and bottom of the $\bm{A}_i$ matrices.  These notions will be made clear by the statement and proof of Theorem~\ref{thm:u}.

\subsection{Preliminaries}
\label{subsec:leap-pre}

We begin by formally defining the simple random leap on the integers.

\begin{defn}\label{defn:leap}
Let $k\ge1$, and let $p_1,\ldots,p_k,q_1,\ldots,q_k$ be nonnegative real numbers such that $\sum_{j=1}^k(p_j+q_j)=1$.
A \emph{simple random leap} is
an integer-valued Markov chain $\{X_n:n\ge0\}$ with transition probabilities for every $n\ge0$ given by
\begin{align*}
P\left(X_{n+1}=j\Given X_n=i\right)=\begin{cases}
p_{j-i}&\text{ if }i\wideplus1\le j\le i\wideplus k,\\
q_{i-j}&\text{ if }i-k\le j\le i-1,\\
0&\text{ if }|i-j|>k \text{ or }i=j.
\end{cases}
\end{align*}
\end{defn}

Note that we use the term ``leap'' since the chain is permitted to ``leap'' over as many as $k-1$ adjacent states to reach a state as many as $k$ units away in a single time step.

\begin{rem}
\citet{cong1982}
considered the special case of such a process where a random walk is simply permitted to have unequal integer-valued step sizes in the positive and negative directions
(i.e., the chain moves $k_1$ units to the right with probability~$p$ and moves $k_2$~units to the left with probability $q=1-p$, where $k_1$ and~$k_2$ are positive integers that need not be equal).
However, it is clear that the  random leap as defined above subsumes both the simple random walk and the model of \citet{cong1982} as very special cases.
\end{rem}

\begin{rem}
We take $P(X_{n+1}=i\given X_n=i)=0$ in Definition~\ref{defn:leap} purely for notational convenience, as the theory may be generalized to allow $P(X_{n+1}=i\given X_n=i)=r\ge0$ with $\sum_{j=1}^k(p_j+q_j)+r=1$.  However, this slight generalization may be straightforwardly analyzed in terms of the original random leap of Definition~\ref{defn:leap} using the standard theory of so-called ``lazy'' Markov chains.
\end{rem}

We will also impose the following assumption throughout the remainder of the paper.

\begin{assum}\label{assum:not-monotone}
$\sum_{j=1}^k p_j>0$ and $\sum_{j=1}^k q_j>0$.
\end{assum}

Assumption~\ref{assum:not-monotone} merely avoids the possibility of a chain that tends monotonically to~$\pm\infty$ with probability~$1$.  With this assumption in place, we now define
the maximum step sizes in the two directions as
$k_p=\max\{k:p_k>0\}$ and $k_q=\max\{k:q_k>0\}$.

The behavior of the simple random leap in regard to recurrence/transience and stationary distributions is discussed in Section~\ref{sec:rtsd}.
We first study absorption probabilities.  To this end, we consider the \emph{absorbing random leap}.
Specifically, we consider a random leap with absorbing barriers at each end of some finite state space $\{0,\ldots,N\}$, which we define as follows.

\begin{defn}\label{defn:absorb}
Let $\{X_n^\star:n\ge0\}$ be a simple random leap
as in Definition~\ref{defn:leap},
and let $N\ge k_p+k_q$ be an integer.  Now let $T^\star=\min\{n\ge0:X_n^\star\le0\text{ or }X_n^\star\ge N\}$, and define a Markov chain $\{X_n:n\ge0\}$ by
\begin{align*}\SADS
X_n=\begin{cases}
X_n^\star &\text{ if }n<T^\star,\\
0 &\text{ if }n\ge T\text{ and }X^\star_{T^\star}\le0,\\
N &\text{ if }n\ge T\text{ and }X^\star_{T^\star}\le N.
\end{cases}
\end{align*}
We call the resulting Markov chain an \emph{absorbing random leap}.
Like the absorbing random jump, the absorbing random leap is also absorbed when it would otherwise ``leap past'' the barrier.
\end{defn}

In lieu of Definition~\ref{defn:absorb}, which in turn depends on Definition~\ref{defn:leap}, we could instead simply define the absorbing random leap as the Markov chain taking values in $\{0,\ldots,N\}$ with transition probabilities for every $n\ge0$ given by
\begin{align*}
P\left(X_{n+1}=j\Given X_n=i\right)=\begin{cases}
p_{j-i} & \text{ if }0<i<N\text{ and }i<j\le \min\{i+k_p,N-1\},\\
\sum_{\ell=N-i}^{k_p} p_\ell &\text{ if }N-k_p\le i<N\text{ and }j=N,\\
q_{i-j} & \text{ if }0<i<N\text{ and }\max\{i-k_q,1\}\le j<i,\\
\sum_{\ell=
i
}^{k_q} q_\ell &\text{ if }0<i\le k_q\text{ and }j=0,\\
1 & \text{ if }i=j=0\text{ or }i=j=N,\\
0 & \text{ otherwise}.
\end{cases}
\end{align*}
For example, when $k_p=k_q=2$, the transition matrix has the form
\newlength\transmatcolwidth
\settowidth{\transmatcolwidth}{$\:p\mathord+p\:$}
\begin{align*}
\bordermatrix{
				&0		&1	&2	&3	&4	&5	&\cdots	&N\mathord{-}3	&N\mathord{-}2	&N\mathord{-}1	&N	\cr
\hfill0\hfill	&1		&0	&0	&0	&0	&0	&\cdots	&0	&0	&0	&0	\cr
\hfill1\hfill
&\makebox[\transmatcolwidth][c]{$q_1\mathord+q_2$}
						&0	&p_1&p_2&0	&0	&\cdots	&0	&0	&0	&0	\cr
\hfill2\hfill	&q_2	&q_1&0	&p_1&p_2&0	&\cdots	&0	&0	&0	&0	\cr
\hfill3\hfill	&0		&q_2&q_1&0	&p_1&p_2&\cdots	&0	&0	&0	&0	\cr
\hfill\vdots\hfill&\vdots&\vdots&\vdots&\vdots&\vdots&\vdots&\ddots&\vdots&\vdots&\vdots&\vdots\cr
\hfill N\mathord{-}3\hfill&0		&0	&0	&0	&0	&0	&\cdots	&0	&p_1	&p_2	&0	\cr
\hfill N\mathord{-}2\hfill&0		&0	&0	&0	&0	&0	&\cdots	&q_1	&0	&p_1	&p_2	\cr
\hfill N\mathord{-}1\hfill&0		&0	&0	&0	&0	&0	&\cdots	&q_2	&q_1	&0	&\makebox[\transmatcolwidth][c]{$p_1\mathord{+}p_2$}\cr
\hfill N\hfill
&\makebox[\transmatcolwidth][c]{$0$}
&\makebox[\transmatcolwidth][c]{$0$}
&\makebox[\transmatcolwidth][c]{$0$}
&\makebox[\transmatcolwidth][c]{$0$}
&\makebox[\transmatcolwidth][c]{$0$}
&\makebox[\transmatcolwidth][c]{$0$}
&\cdots
&\makebox[\transmatcolwidth][c]{$0$}
&\makebox[\transmatcolwidth][c]{$0$}
&\makebox[\transmatcolwidth][c]{$0$}
&1
}.
\end{align*}
Although
this definition of the absorbing random leap
is perhaps more elementary, it somewhat obscures the connection between the absorbing random leap and the non-absorbing simple random leap on~$\mathbb{Z}$.

We now consider the long-term behavior of the absorbing simple random leap.
Some of its properties are rather trivial.  For instance, the
absorbing states $0$ and $N$ are recurrent while all other states are transient, and a distribution~$\bm{\pi}$ is stationary if and only if it satisfies $\pi_0+\pi_N=1$.
Instead of these properties, we are primarily interested in
events associated with the absorption of the chain into one of the absorbing
barriers.
Note that these absorption events may alternatively be interpreted in terms of first passage times for a non-absorbing simple random leap.  The analogous interpretation for simple random walks was discussed in Subsection~\ref{subsec:walk}.
Probabilities and expected values associated with these events are the focus of Subsections~\ref{subsec:leap-u}~and~\ref{subsec:leap-v}.

\subsection{Connections to Other Models}
\label{subsec:leap-graph}

The random leap
has some similarities to
to another generalization of the random walk: the random walk on a graph.  Let $G$ be a graph with vertex set~$S$.  A random walk on~$G$ is a Markov chain taking values in~$S$ defined so that if $X_n=x$, then $X_{n+1}\given X_n=x$ has the discrete uniform distribution on the set of vertices adjacent to~$x$.  (Clearly every vertex must be adjacent to at least one other vertex for such a distribution to be well-defined.)  Now suppose $S=\mathbb{Z}$, and suppose $G$ is the graph
where every pair of distinct vertices~$i$ and~$j$ are adjacent if and only if $|i-j|\le k$.
The $k=2$ case of this graph is illustrated in Figure~\ref{fig:jump-graph}.
Then this special case of the random walk on a graph may also be regarded as a special case of the simple random leap on~$\mathbb{Z}$ where $p_j=q_j=1/2k$ for each $j\in\{1,\ldots,k\}$.
It should be recognized, however, that neither the random leap nor the random walk on a graph can be wholly viewed as a special case of the other.
Specifically, the random leap also admits unequal transition probabilities, while the random walk on a graph admits other types of graphs beyond simply the ``$2k$ nearest neighbors'' graph defined above.

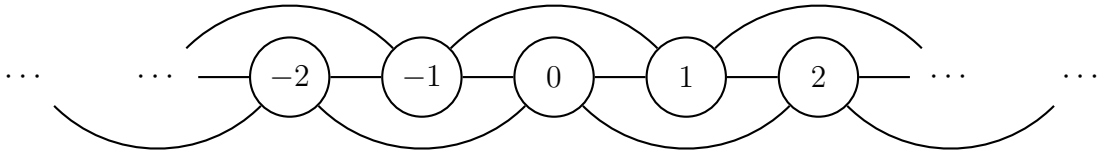
\begin{figure}[htbp]
\color{black}
\begin{center}
\begin{tikzpicture}[node distance=50pt,thick]
\node[circle,draw,minimum size=30pt] (n2) {$-2$};
\node[circle,draw,minimum size=30pt] (n1) [right of=n2] {$-1$};
\node[circle,draw,minimum size=30pt] (0) [right of=n1] {$0$};
\node[circle,draw,minimum size=30pt] (1) [right of=0] {$1$};
\node[circle,draw,minimum size=30pt] (2) [right of=1] {$2$};
\node[circle,minimum size=30pt] (3) [right of=2] {$\cdots$};
\node[circle,minimum size=30pt] (4) [right of=3] {$\cdots$};
\node[circle,minimum size=30pt] (n3) [left of=n2] {$\cdots$};
\node[circle,minimum size=30pt] (n4) [left of=n3] {$\cdots$};
\path
(2) edge node {} (3)
(1) edge node {} (2)
(0) edge node {} (1)
(n1) edge node {} (0)
(n2) edge node {} (n1)
(n3) edge node {} (n2)
(4) edge [bend left, out=45, in=135] node {} (2)
(2) edge [bend left, out=45, in=135] node {} (0)
(0) edge [bend left, out=45, in=135] node {} (n2)
(n2) edge [bend left, out=45, in=135] node {} (n4)
(3) edge [bend right, out=315, in=225] node {} (1)
(1) edge [bend right, out=315, in=225] node {} (n1)
(n1) edge [bend right, out=315, in=225] node {} (n3)
;
\end{tikzpicture}
\end{center}
\vspace*{-2\topsep}
\caption{Graph with vertex set $S=\mathbb{Z}$ where distinct vertices $i$ and~$j$ are adjacent if and only if $|i-j|\le2$.}
\label{fig:jump-graph}
\end{figure}

\subsection{Absorption Probabilities}
\label{subsec:leap-u}

For the absorbing random leap, we first consider 
$u_i=P(X_T=N\given X_0=i)$, the upper absorption probability from state~$i$,
where $T=\min\{n\ge0:X_n=0\text{ or }X_n=N\}$.
Recall that $u_i$ may alternatively be interpreted in terms of first passage times for a non-absorbing simple random leap.
Trivially, we may immediately note that $u_0=0$ and $u_N=1$.

The conventional approach to finding absorption probabilities from transient starting 
states 
is
to partition the transition matrix as
\[
\bm{P}=\begin{pmatrix}\bm{Q}&\bm{R}\\\bm{0}&\bm{P}_\star\end{pmatrix},
\]
where $\bm{Q}$ and $\bm{P}_\star$ are square and correspond to the transient and absorbing states, respectively
\citep{resnick1992}.
Then the $(i,j)$th element of the matrix $(\bm{I}-\bm{Q})^{-1}\bm{R}$ gives the probability of absorption by the $j$th absorbing state when starting in the $i$th transient state.
This form may appear useful, but it has a number of shortcomings, which we explain below.
First, the $(\bm I-\bm Q)^{-1}\bm R$ method does not provide any tractable closed-form expression for a single absorption probability~$u_i$ (except of course as the $i$th element of a much larger vector).
Thus, $(\bm I-\bm Q)^{-1}\bm R$ is merely a symbolic expression that does not provide a closed form for its individual elements, whereas the method we develop in this subsection accomplishes precisely this task.
Second, the number of rows and columns of the matrix $\bm I-\bm Q$ is $N-1$, the number of states enclosed by the two barriers.  This number is typically large (especially in ``big data'' applications), so inversion of the matrix $\bm I-\bm Q$ may be quite unwieldy since it is of order $O(N^3)$ computationally.  In contrast, the method we develop in this subsection requires only the computation of determinants of matrices with $2k-1$ rows and columns, which is of order $O(k^3)$ computationally.  Since $k$ is typically much smaller than~$N$, our method is thus considerably more efficient.
Third, the $(\bm I-\bm Q)^{-1}\bm R$ method does not provide much insight into the derivation of related quantities (e.g., expected absorption times or stationary distributions) for this chain and similar chains (e.g., random leaps with reflecting barriers).  In contrast, the method we develop in this subsection introduces a framework by which several of these related tasks can also be accomplished.

Before stating our results,
we first
provide a definition.
\begin{defn}
The \emph{characteristic polynomial} of
the
random leap is
\[
\chi(z)=\sum_{j=0}^{\mathclap{2k-1}} c_j\,z^{j},
\]
where the coefficients
$c_{j}$
are given by
\[
c_{j}=\begin{cases}
-\sum_{\ell=k-j}^k q_\ell & \text{ if }0\le j\le k-1,\\
\phantom{-}\sum_{\ell=j-k+1}^k p_\ell & \text{ if }k\le j\le 2k-1.
\end{cases}
\]
\end{defn}

\begin{rem}
\label{rem:chi}
The characteristic polynomial may instead be written as $\chi(z)=\sum_{j=k-k_q}^{k+k_p-1}c_j\,z^j$ to exclude any leading and trailing terms with zero coefficients.
\end{rem}

We now consider the roots of~$\chi(z)$.  From the form in Remark~\ref{rem:chi},
it is clear that $\chi(z)$ is a polynomial of order $k+k_p-1$ and hence has $k+k_p-1$ roots (counted according to multiplicity).
Now observe that
\bas\SADS
\chi(z)=z^{k-k_q}\smashoperator{\sum_{j=k-k_q}^{k+k_p-1}}c_jz^{j-k+k_q}.
\eas
Then since $c_{k-k_q}\ne0$, it is clear that zero is a root of $\chi(z)$ if and only if $k_q<k$.  Furthermore, if zero is a root of $\chi(z)$, then its multiplicity is $k-k_q$.
Thus, the number of nonzero roots of~$\chi(z)$ (again counted according to multiplicity) is $(k+k_p-1)-(k-k_q)=k_p+k_q-1$ (i.e.,~the total number of roots minus the number of times that zero occurs as a root).
Now let $z_1,\ldots,z_s$ denote the distinct nonzero roots of~$\chi(z)$, and let $ r_1,\ldots, r_s$ denote their respective multiplicities.  Also let $r=\sum_{j=1}^s r_j=k_p+k_q-1$.
Then for each $j\in\{1,\ldots,s\}$, let~$\bm{Z}^{(j)}$ be the $(N+ r-1)\times r_j$ matrix with elements
\[
Z^{(j)}_{i,\ell}=(i+1-k_q)^{\ell-1}z_j^{i+1-k_q},
\]
with the convention that $0^0=1$.
Next, let $\bm{Z}$ be the $(N+ r-1)\times r$ matrix
\begin{align*}
\bm{Z}=\begin{pmatrix}\bm{Z}^{(1)}&\bm{Z}^{(2)}&\cdots&\bm{Z}^{(s)}\end{pmatrix},
\end{align*}
which has full column rank.

\begin{rem}
Some of the roots $z_1,\ldots,z_s$ may be complex, but this poses no problem since
appropriate boundary conditions
will ensure that our eventual results yield real numbers.
Since the possibly complex nature of these roots will seldom be relevant,
we will continue to use $i$ to denote a generic index, and never $\sqrt{-1}$.
\end{rem}

We now define matrices that will be of importance
for stating our results concisely.
Note that we will henceforth write $\bm{0}$ to denote a zero matrix or vector if its size is contextually clear.  We will also write $\bm{I}_m$ to denote the $m\times m$ identity matrix for any nonnegative integer~$m$, where we interpret $\bm{I}_0$ as an empty matrix.
\begin{defn}\label{defn:accordion}
For each $i\in\{0,\ldots,N\}$, the $i$th \emph{accordion matrix} $\bm{A}_i$ is the $ r\times(N+ r-1)$ matrix
\bas\SADS
\bm{A}_i=\begin{pmatrix}
\bm{I}_{k_q-1}&\bm{0}^{\phantom{T}}&\bm{0}_{\phantom{N-i}}&\bm{0}\\
\bm{0}
&\bm{1}_i^\T&\bm{0}_{N-i}^\T&
\bm{0}\\
\bm{0}&\bm{0}^{\phantom{T}}&\bm{0}_{\phantom{N-i}}&\bm{I}_{k_p-1}\end{pmatrix}.
\eas
\end{defn}

The motivation for the name ``accordion'' is made clear by considering the form of $\bm{A}_i\bm{M}$ for any matrix $\bm{M}$ with $N+ r-1$ rows: $\bm{A}_i\bm{M}$ retains the top $k_q-1$ rows and bottom $k_p-1$ rows of~$\bm{M}$ while summing the first $i$ of the $N$ middle rows of $\bm{M}$ into a single row.
With this notation, we now state and prove the
following
result.

\begin{thm}\label{thm:u}
Let $\{X_n:n\ge0\}$ be an absorbing random leap.  For each $i\in\{0,\ldots,N\}$, let $u_i=P(X_n=N\text{ for some }n\ge0\given X_0=i)$ be the upper absorption probability from starting state~$i$.
Then
\[
u_i=\frac{\det(\bm{A}_{\mathrlap{\,i}\hphantom{N}}\bm{Z})}{\det(\bm{A}_N\bm{Z})}
\]
for each $i\in\{0,\ldots,N\}$.
\end{thm}

\begin{proof}
Begin by taking $u_i=0$ for any $i<0$ and $u_i=1$ for any $i>N$,
noting that these values can be fully justified by viewing the absorbing random leap as simply a ``stopped'' version of the simple random leap on~$\mathbb{Z}$.
Then define $d_i=u_i-u_{i-1}$ for every~$i$, noting that $d_i\ne0$ only if $i\in\{1,\ldots,N\}$.
Then for each $i\in\{k_p+1,\ldots,N+k_p-1\}$, a first-step analysis yields
\begin{align*}\SADS
u_{i-k_p}&
=\sum_{j=0}^N P\left(X_T=N\Given X_0=i-k_p,\;X_1=j\right)\;P\left(X_1=j\Given X_0=i-k_p\right)\\
&
=\sum_{j=0}^N P\left(X_T=N\Given X_0=j\right)\;P\left(X_1=j\Given X_0=i-k_p\right)
=\sum_{\ell=1}^{k_p} p_\ell\,u_{i-k_p+\ell}+\sum_{\ell=1}^{k_q} q_\ell\,u_{i-k_p-\ell},
\end{align*}
noting that the last equality holds for $i\in\{k_p+1,\ldots, r-1\}$ because we have taken $u_i=0=u_0$ when $i<0$, and that it holds for $i\in\{N+1,\ldots,N+k_p-1\}$ because we have taken $u_i=1=u_N$ when $i>N$.  Then
\begin{align*}\SADS
0&=\sum_{\ell=1}^{k_p} p_\ell\, u_{i-k_p+\ell}+\sum_{\ell=1}^{k_q} q_\ell\, u_{i-k_p-\ell}-\left(\sum_{\ell=1}^{k_p} p_\ell+\sum_{\ell=1}^{k_q} q_\ell\right)u_{i-k_p}\\
&=\sum_{\ell=1}^{k_p} p_\ell \left(u_{i-k_p+\ell}-u_{i-k_p}\right)+\sum_{\ell=1}^{k_q} q_\ell \left(u_{i-k_p-\ell}-u_{i-k_p}\right)\\
&=\sum_{\ell=1}^{k_p} p_\ell \left(\sum_{j=1}^\ell d_{i-k_p+j}\right)+\sum_{\ell=1}^{k_q} q_\ell \left(-\sum_{j=1}^\ell d_{i-k_p-j+1}\right)\\
&=\sum_{j=1}^{k_p}\left(\sum_{\ell=j}^{k_p} p_\ell\right)d_{i-k_p+j}-\sum_{j=1}^{k_q}\left(\sum_{\ell=j}^{k_q} q_\ell\right)d_{i-k_p-j+1},
\end{align*}
which implies that
\basn
d_i&=\sum_{j=1}^{\mathclap{k_p-1}}\left(-\frac{1}{p_{k_p}}\sum_{\ell=j}^{k_p} p_\ell\right)d_{i-k_p+j}+\sum_{j=1}^{k_q}\left(\frac{1}{p_{k_p}}\sum_{\ell=j}^{k_q} q_\ell\right)d_{i-k_p-j+1}\notag\\
&=-\frac{1}{p_{k_p}}\sum_{j=1}^{ r}c_{k-k_q+j-1}\,d_{i-j}\label{d-recur-hom}
\easn
for each $i\in\{k_p+1,\ldots,N+k_p-1\}$.  Thus, the $d_i$ are governed by a linear recurrence relation of order~$ r$ with characteristic polynomial
\begin{align}
\chi^\star(z)=\frac{1}{p_{k_p}}\sum_{j=0}^{ r}c_{k-k_q+j}\,z^{j}
=\frac{z^{k_q-k}}{p_{k_p}}\sum_{j=k\mathrlap{-k_q}}^{k+\mathrlap{k_p-1}}c_{j}\,z^{j}
=\frac{z^{k_q-k}}{p_{k_p}}\;\chi(z),
\label{chi-star}
\end{align}
where the last equality is justified by Remark~\ref{rem:chi}.
Thus, the $ r$ roots of $\chi^\star(z)$ coincide with the $ r$ nonzero roots of $\chi(z)$.
The general form of the solution to
the recurrence relation in~(\ref{d-recur-hom})
is
\[
d_i=\sum_{j=1}^s\sum_{\ell=1}^{ r_j}b^\star_{j,\ell}\,i^{\ell-1}z_j^i,
\]
with the convention that $0^0=1$,
and where the $b^\star_{j\ell}$ are coefficients that must be found through boundary conditions, of which $ r$ are needed.  If we define vectors
$\bm{b}=(b_1,\ldots,b_{ r})=(b^\star_{1,1},\ldots,b^\star_{s,r_s})$
and
$\bm{d}=(d_{2-k_q},d_{3-k_q},\ldots,d_{N+k_p-1})$,
then we may write the general solution as simply $\bm{d}=\bm{Z}\bm{b}$.  Now let the $ r$ required boundary conditions be
\bas\SADS
d_i=0\quad\text{ for each }i\in\{2-k_q,\ldots,0\}\cup\{N+1,\ldots,N+k_p-1\},
\qquad
\sum_{i=1}^N d_i=u_N-u_0=1,
\eas
which may be represented in matrix form as $\bm{A}_N\bm{d}=\bm{e}_{k_q}$, where $\bm{e}_{k_q}$ is the $k_q$th unit vector of length~$ r$.  Then $\bm{A}_N\bm{Z}\bm{b}=\bm{e}_{k_q}$.
We now establish that the matrix $\bm A_N\bm Z$ is nonsingular.
Suppose instead that $\bm A_N\bm Z$ is singular.  Then the equation $\bm A_N\bm Z\bm b=\bm e_{k_q}$ has either no solutions for~$\bm b$ or infinitely many solutions for~$\bm b$.  If there are no solutions for~$\bm b$, then there are no solutions for $\bm d=\bm Z\bm b$.  If instead there are infinitely many solutions for~$\bm b$, then there are also infinitely many solutions for~$\bm d$ since $\bm Z$ has full column rank.  Clearly both cases lead to a contradiction since there must be a unique solution for~$\bm d$.
Thus, $\bm A_N\bm Z$ is nonsingular.
We may therefore
apply Cramer's rule and the Laplace expansion of the determinant to write the elements $b_\ell$ of $\bm{b}$ as
\bas\SADS
b_\ell=\frac{\det(
\bm{\Lambda}_\ell
)}{\det(\bm{A}_N\bm{Z})}=(-1)^{k_q+\ell}\;\frac{\det(\tilde{\bm{A}}_{k_q}\tilde{\bm{Z}}_\ell)}{\det(\bm{A}_{\mathrlap{N}\hphantom{k_q}}\bm{Z}_{\hphantom\ell})},
\eas
where $\bm{\Lambda}_\ell$ is the matrix formed by replacing the $\ell$th column of $\bm{A}_N\bm{Z}$ with $\bm{e}_{k_q}$,
and where $\tilde{\bm{A}}_\ell$ and $\tilde{\bm{Z}}_\ell$ denote the matrices formed by deleting the $\ell$th row of $\bm{A}_N$ and the $\ell$th column of $\bm{Z}$, respectively.
Then the $j$th element $(\bm{d})_j=d_{1-k_q+j}$ of $\bm{d}$ is
\begin{align}\label{numerator-laplace}
d_{1-k_q+j}
=\frac{1}{\det(\bm{A}_N\bm{Z})}\sum_{\ell=1}^{ r}Z_{j,\ell}\;(-1)^{k_q+\ell}
\det(\tilde{\bm{A}}_{k_q}\tilde{\bm{Z}}_{\ell}),
\end{align}
where $Z_{j,\ell}$ is the $(j,\ell)$th element of $\bm{Z}$.
Now, for each $j\in\{1,\ldots, r\}$, define $\breve{\bm{A}}_{j}$ to be the $ r\times(N+ r-1)$ matrix formed by replacing the $k_q$th row of $\bm{A}_N$ with~$\tilde{\bm{e}}_j^\T$, where $\tilde{\bm{e}}_j$ is the $j$th unit vector of length $N+r-1$.
Then the sum in~(\ref{numerator-laplace}) is precisely the Laplace expansion of $\det(\breve{\bm{A}}_j\bm{Z})$ along its $k_q$th row, and thus
$d_{1-k_q+j}=\det(\breve{\bm{A}}_j\bm{Z})/\det(\bm{A}_N\bm{Z})$ for each $j\in\{1,\ldots,N+ r-1\}$.
Finally, for each $i\in\{0,\ldots,N\}$, we have
\begin{align*}
u_i=\sum_{j=1}^i d_j
=\smashoperator{\sum_{j=k}^{i+k_q-1}}d_{1-k_q+j}
=\frac{1}{\det(\bm{A}_N\bm{Z})}\smashoperator{\sum_{j=k_q}^{i+k_q-1}}\det(\breve{\bm{A}}_j\bm{Z})
&
=
\frac{\det(\bm{A}_{\mathrlap{\,i}\hphantom{N}}\bm{Z})}{\det(\bm{A}_N\bm{Z})},
\end{align*}
noting that the matrices $\breve{\bm{A}}_{k_q},\ldots,\breve{\bm{A}}_{i+k_q-1}$ differ only in the $k_q$th row and that the sum of these rows is the $k_q$th row of $\bm{A}_i$.
\end{proof}


\begin{cor}[to Theorem~\ref{thm:u}]
The upper absorption probabilities for the simple random walk and simple random jump
are given by the expressions
in Lemma~\ref{lem:u-1} and Theorem~\ref{thm:u-2}.
\end{cor}

\begin{proof}
The results follow immediately from the forms of $\bm{Z}$ and $\bm{A}_i$.
\end{proof}

\subsection{Expected Absorption Times}
\label{subsec:leap-v}

We now shift our attention to the expected absorption time from starting state~$i$, which we denote by
$v_i=E(T\given X_0=i)$, where once again $T=\min\{n\ge0:X_n=0\text{ or }X_n=N\}$.
Note that $v_i$ may alternatively be interpreted as the expected first passage time of the set $S_0\cup S_1=\mathbb{Z}\setminus\{1,\ldots,N-1\}$ for a non-absorbing simple random leap with initial state~$i$.  Trivially, we may immediately note that $v_0=v_N=0$.
The
form of the remaining values is provided by the following result, which also includes $v_0$ and $v_N$ as special cases.
We
first introduce some additional notation.  Let
\begin{align*}
\mu&=\sum_{j=1}^k j(p_j-q_j)=E\left(X_{n+1}-X_n\Given k_q\le X_n\le N-k_p\right),\\
\sigma^2&=\sum_{j=1}^k j^2(p_j+q_j)-\mu^2=\Var\left(X_{n+1}-X_n\Given k_q\le X_n\le N-k_p\right),
\end{align*}
and define a vector
$\bm{\delta}=(\delta_{2-k_q}\;\;\delta_{3-k}\;\;\cdots\;\;\delta_{N+k_p-1})^\T$ 
according to
\[
\delta_i=\begin{cases}-1/\mu & \text{ if }\mu\ne0,\\
-2i/\sigma^2 & \text{ if }\mu=0,\end{cases}
\]
for each $i\in\{2-k_q,\ldots,N+k_p-1\}$.
We now provide another definition.
\begin{defn}
For each $i\in\{0,\ldots,N\}$, the $i$th \emph{extended accordion matrix} $\bm{A}^{\star}_i$ is the $( r+1)\times(N+ r-1)$ matrix
\[
\bm{A}^{\star}_i
=\begin{pmatrix}
\bm{I}_{k_q-1}&\bm{0}^{\phantom{T}}&\bm{0}_{\phantom{N-i}}&\bm{0}\\
\bm{0}
&\bm{1}_i^\T&\bm{1}_{N-i}^\T&
\bm{0}\\
\bm{0}&\bm{0}^{\phantom{T}}&\bm{0}_{\phantom{N-i}}&\bm{I}_{k_p-1}\\
\bm{0}&\bm{1}_i^\T&\bm{0}_{N-i}^\T&\bm{0}\end{pmatrix}
=\begin{pmatrix}\bm{A}_N\\ \bm{e}_{k_q}^\T\bm{A}_i\end{pmatrix},
\]
where $\bm{e}_{k_q}$ is the $k_q$th unit vector of length~$r$.
\end{defn}
Finally, let $\bm{Z}^{\star}$ be the $(N+ r-1)\times( r+1)$ matrix given by $\bm{Z}^{\star}=(\bm{Z}\;\;\bm{\delta})$.  Then we have the following result.

\begin{thm}\label{thm:v}
Let $\{X_n:n\ge0\}$ be an absorbing random leap.  For each $i\in\{0,\ldots,N\}$, let $v_i$ be the expected absorption time from starting state~$i$.
Then
\[
v_i=\frac{\det(\bm{A}^{\,\star}_{\mathrlap{\,i}\hphantom{N}}\bm{Z}^{\star})}{\det(\bm{A}_N\bm{Z}^{\hphantom\star})}\
\]
for each $i\in\{0,\ldots,N\}$.
\end{thm}

\begin{proof}
Begin by taking $v_i=0$ for any $i\notin\{0,\ldots,N\}$,
noting that these values can be fully justified by viewing the absorbing random leap as simply a ``stopped'' version of the simple random leap on~$\mathbb{Z}$.
Then define $d_i=v_i-v_{i-1}$ for every~$i$, noting that $d_i\ne0$ only if $i\in\{1,\ldots,N\}$.  Then for each $i\in\{k_p+1,\ldots,N+k_p-1\}$, a first-step analysis yields
\begin{align*}
v_{i-k_p}&=\sum_{j=0}^N E\left(X_T=N\Given X_0=i-k_p,\;X_1=j)\;P(X_1=j\mid X_0=i-k_p\right)\\
&=\sum_{j=0}^N \left[1+E\left(X_T=N\Given X_0=j\right)\right]\;P\left(X_1=j\Given X_0=i-k_p\right)\\
&=1+\sum_{\ell=1}^{k_p} p_\ell\, v_{i-k_p+\ell}+\sum_{\ell=1}^{k_q} q_\ell\, v_{i-k_p-\ell},
\end{align*}
noting that the last equality holds for $i\in\{k_p+1,\ldots, r-1\}$ because we have taken $v_i=0=v_0$ when $i<0$, and that it holds for $i\in\{N+1,\ldots,N+k_p-1\}$ because we have taken $v_i=0=v_N$ when $i>N$.  Then
\begin{align*}\SADS
0&=1+\sum_{\ell=1}^{k_p} p_\ell\, v_{i-k_p+\ell}+\sum_{\ell=1}^{k_q} q_\ell\, v_{i-k_p-\ell}-\left(\sum_{\ell=1}^{k_p} p_\ell+\sum_{\ell=1}^{k_q} q_\ell\right)v_{i-k_p}\\
&=1+\sum_{\ell=1}^{k_p} p_\ell \left(v_{i-k_p+\ell}-v_{i-k_p}\right)+\sum_{\ell=1}^{k_q} q_\ell \left(v_{i-k_p-\ell}-v_{i-k_p}\right)\\
&=1+\sum_{\ell=1}^{k_p} p_\ell \left(\sum_{j=1}^\ell d_{i-k_p+j}\right)+\sum_{\ell=1}^{k_q} q_\ell \left(-\sum_{j=1}^\ell d_{i-k_p+1-j}\right)\\
&=1+\sum_{j=1}^{k_p}\left(\sum_{\ell=j}^{k_p} p_\ell\right)d_{i-k_p+j}-\sum_{j=1}^{k_q}\left(\sum_{\ell=j}^{k_q} q_\ell\right)d_{i-k_p-\ell+1},
\end{align*}
which implies that
\basn\SADS
d_i&=-\frac{1}{p_{k_p}}+\smashoperator{\sum_{j=1}^{k_p-1}}\left(-\frac{1}{p_{k_p}}\sum_{\ell=j}^{k_p} p_\ell\right)d_{i-k_p+j}+\sum_{j=1}^{k_q}\left(\frac{1}{p_{k_p}}\sum_{\ell=j}^{k_q} q_\ell\right)d_{i-k_p-\ell+1}\notag\\
&=-\frac{1}{p_{k_p}}-\frac{1}{p_{k_p}}
\sum_{j=1}^{ r}c_{k-k_q+j-1}\,d_{i-j}\label{d-recur-inhom}
\easn
for each $i\in\{k_p+1,\ldots,N+k_p-1\}$.  Thus, the $d_i$ are governed by an inhomogeneous linear recurrence relation of order $ r$ with characteristic polynomial $\chi^\star(z)$ as given in the proof of Theorem~\ref{thm:u}.  The general form of the solution to
the recurrence relation in~(\ref{d-recur-inhom})
may be written as the sum of a particular solution and the solution to the corresponding homogeneous linear recurrence relation.
We now verify that the particular solution is~$\bm{\delta}$.
First, if $\mu\ne0$, then $\delta_i=-1/\mu$ for each~$i$, and
\begin{align*}
-\frac{1}{p_{k_p}}-\frac{1}{p_{k_p}}\sum_{j=1}^{ r}\left(-\frac{
c_{k-k_q+j-1}
}{\mu}\right)
&=
-\frac{1}{p_{k_p}}-\frac{c_{k+k_p-1}}{p_{k_p}\mu}+\frac{1}{p_{k_p}\mu}
\sum_{j=0}^{ r}c_{k-k_q+j}\\
&=
-\frac{1}{p_{k_p}}-\frac{1}{\mu}+\frac{1}{p_{k_p}\mu}\left(\sum_{j=1\;\vphantom{k}}^{k_p}\sum_{\ell=\mathrlap{k_p+1-j}\;\;\;}^{k_p} p_\ell\;\;\;-\;\;\;\sum_{\;\;\mathllap{j=k_p}+1}^{ r+1}\sum_{\ell=j\mathrlap{-k_p}\;}^{k_q} q_\ell\right)\\
&=
-\frac{1}{p_{k_p}}-\frac{1}{\mu}+\frac{1}{p_{k_p}\mu}\sum_{j=1}^{k}j(p_j-q_j)=-\frac{1}{\mu}.
\end{align*}
If instead $\mu=0$, then $\delta_i=-2i/\sigma^2$
for each~$i$,
and
\begin{align*}
&-\frac{1}{p_{k_p}}-\frac{1}{p_{k_p}}\sum_{j=1}^{ r}\left[-\frac{2(i-j)\,
c_{k-k_q+j-1}
}{\sigma^2}\right]\\
&\qquad=-\frac{1}{p_{k_p}}-\frac{2i\,c_{k+k_p-1}}{p_{k_p}\sigma^2}+\frac{2}{p_{k_p}\sigma^2}
\sum_{j=0}^{ r}(i-j-1)\,c_{k-k_q+j-1}\\
&\qquad=
-\frac{1}{p_{k_p}}-\frac{2i}{\sigma^2}+\frac{2(i-1)\mu}{p_{k_p}\sigma^2}-\frac{2}{p_{k_p}\sigma^2}
\sum_{j=0}^{ r}j\,c_{k-k_q+j-1}\\
&\qquad=-\frac{1}{p_{k_p}}-\frac{2i}{\sigma^2}-\frac{1}{p_{k_p}\sigma^2}\sum_{\ell=1}^{k_p}\ell(2k_p-\ell+1)p_\ell
+\frac{1}{p_{k_p}\sigma^2}\sum_{\ell=1}^{k_q}\ell(2k_p+\ell+1)q_\ell\\
&\qquad=-\frac{1}{p_{k_p}}-\frac{2i}{\sigma^2}-\frac{(2k_p+1)\mu}{p_{k_p}\sigma^2}
+\frac{\sigma^2+\mu^2}{p_{k_p}\sigma^2}\\
&\qquad=-\frac{1}{p_{k_p}}-\frac{2i}{\sigma^2}+\frac{1}{p_{k_p}}=-\frac{2i}{\sigma^2}.
\end{align*}
Thus,
$\bm{\delta}$ is indeed the particular solution in both cases.
Then we may write
\[
d_i=\delta_i+\sum_{j=1}^s\sum_{\ell=1}^{r_j}b^\star_{j,\ell}\,i^{\ell-1}z_j^i,
\]
with the convention that $0^0=1$, and where the $b^\star_{j\ell}$ are coefficients that must be found through boundary conditions, of which $ r$ are needed.  If we again define vectors
$\bm{b}=(b_1,\ldots,b_{ r})=(b^\star_{1,1},\ldots,b^\star_{s,r_s})$
and  $\bm{d}=(d_{2-k_q},d_{3-k_q},\ldots,d_{N+k_p-1})$,
then we may write the general solution as simply $\bm{d}=\bm{\delta}+\bm{Z}\bm{b}$.  Now let the $ r$ required boundary conditions be
\[
d_i=0\quad\text{ for each }i\in\{2-k_q,\ldots,0\}\cup\{N+1,\ldots,N+k_p-1\},
\qquad
\sum_{i=1}^N d_i=v_N-v_0=0,
\]
which may be represented in matrix form as $\bm{A}_N\bm{d}=\bm{0}$.  
Then $\bm{A}_N\bm{\delta}+\bm{A}_N\bm{Z}\bm{b}=\bm{0}$, and hence $\bm{A}_N\bm{Z}\bm{b}=-\bm{A}_N\bm{\delta}$.
We may apply Cramer's rule to write the elements~$b_\ell$ of~$\bm{b}$ as
\begin{align*}
b_\ell=\frac{\det(\bm{A}_N\breve{\bm{Z}}_\ell)}{\det(\bm{A}_N\bm{Z}_{\hphantom\ell})}
=\frac{1}{\det(\bm{A}_N\bm{Z})}\begin{vmatrix}\bm{A}_N\bm{Z}&\bm{A}_N\bm{\delta}\\ \bm{e}_\ell^\T&0\end{vmatrix},
\end{align*}
where $|\cdot|$ denotes the determinant, $\bm{e}_\ell$ is the $\ell$th unit vector of length~$ r$, and $\breve{\bm{Z}}_\ell$ is the matrix formed by replacing the $\ell$th column of~$\bm{Z}$ with~$-\bm{\delta}$.
Then let $Z_{j,\ell}$ denote the $(j,\ell)$th element of~$\bm{Z}$ and write
the $j$th element $(\bm{d})_j=d_{1-k_q+j}$ of~$\bm{d}$ as
\begin{align}
d_{1-k_q+j}&=\delta_{1-k_q+j}+\frac{1}{\det(\bm{A}_N\bm{Z})}\sum_{\ell=1}^{ r}Z_{j,\ell}\begin{vmatrix}\bm{A}_N\bm{Z}&\bm{A}_N\bm{\delta}\\ \bm{e}_\ell^\T&0\end{vmatrix}\label{clever}\\
&=\tilde{\bm{e}}_j^\T\bm{\delta}+\frac{1}{\det(\bm{A}_N\bm{Z})}
\begin{vmatrix}
\bm{A}_N\bm{Z}&\bm{A}_N\bm{\delta}\\
\hphantom{\bm A_N}\mathllap{\tilde{\bm{e}}_j^\T}\bm{Z}&0
\end{vmatrix}
=\frac{1}{\det(\bm{A}_N\bm{Z})}
\begin{vmatrix}
\bm{A}_N\bm{Z}&\bm{A}_N\bm{\delta}\\
\hphantom{\bm A_N}\mathllap{\tilde{\bm{e}}_j^\T}\bm{Z}&\hphantom{\bm A_N}\mathllap{\tilde{\bm{e}}_j^\T}\bm{\delta}
\end{vmatrix}
\notag
\end{align}
where $\tilde{\bm{e}}_j$
is the $j$th unit vector of length $N+ r-1$.
Finally, for each $i\in\{0,\ldots,N\}$, we have
\begin{align*}
v_i=\sum_{j=1}^i d_j=\smashoperator{\sum_{j=k_q}^{i+k_q-1}d_{j-k_q+1}}
&=\frac{1}{\det(\bm{A}_N\bm{Z})}\smashoperator{\sum_{j=k_q}^{i+k_q-1}}\;\;
\begin{vmatrix}
\bm{A}_N\bm{Z}&\bm{A}_N\bm{\delta}\\
\hphantom{\bm A_N}\mathllap{\tilde{\bm{e}}_j^\T}\bm{Z}&\hphantom{\bm A_N}\mathllap{\tilde{\bm{e}}_j^\T}\bm{\delta}
\end{vmatrix}
\\
&=\frac{1}{\det(\bm{A}_N\bm{Z})}
\begin{vmatrix}
\hphantom{\bm{e}_{k_q}^\T}\bm{A}_N\bm{Z}&\hphantom{\bm{e}_{k_q}^\T}\bm{A}_N\bm{\delta}\\
\bm{e}_{k_q}^\T\bm{A}_{\mathrlap{\,i}\hphantom{N}}\bm{Z}&\bm{e}_{k_q}^\T\bm{A}_{\mathrlap{\,i}\hphantom{N}}\bm{\delta}
\end{vmatrix}
=\frac{\det(\bm{A}_{\mathrlap{\,i}\hphantom{N}}^{\,\star}\bm{Z}^{\star})}{\det(\bm{A}_N\bm{Z}^{\hphantom\star})}
\end{align*}
since $\sum_{j=k_q}^{i+k_q-1}\tilde{\bm{e}}_j^\T$ is precisely the $k_q$th row of $\bm{A}_i$.
\end{proof}

\begin{rem}
One of the remarkable features of the calculations from~(\ref{clever}) to the end of the proof is the ability to express the expected absorption times~$v_i$ very similarly to the upper absorption probabilities~$u_i$ obtained in Theorem~\ref{thm:u}.
\end{rem}

As a corollary to Theorem~\ref{thm:v}, we now derive the well-known result for the expected absorption time for the simple random walk \citep[e.g.,][]{feller1968}.

\begin{cor}[to Theorem~\ref{thm:v}]
For the simple random walk
with $p>0$,
the expected absorption times are
\begin{align*}\SADS
v_i=\begin{cases}i(N-i) & \text{ if }p=1/2,\\
\displaystyle\frac{1}{1-2p}\left[i-\frac{N(1-z_1^i)}{1-z_1^N}\right] & \text{ if }p\ne1/2,
\end{cases}
\end{align*}
where $z_1=(1-p)/p$.
\end{cor}
\begin{proof}
Note that in the $k_p=k_q=1$ case, we again have $\det(\bm{A}_N\bm{Z})=\sum_{j=1}^Nz_1^j$, while
\begin{align*}
\det(\bm{A}_i^{\star}\bm{Z}^{\star})=\begin{vmatrix}
\sum_{j=1}^Nz_1^j & \sum_{j=1}^N\delta_j\\
\sum_{j=1}^iz_1^j & \sum_{j=1}^i\delta_j
\end{vmatrix}.
\end{align*}
If $p=1/2$, then $\mu=0$ and $\sigma^2=1$, so $\delta_i=-2i$.  Then since $(1-p)/p=1$, we have
\begin{align*}
\det(\bm{A}_i^{\star}\bm{Z}^{\star})=\begin{vmatrix}
N & -N(N+1)\\
i & -i(i+1)
\end{vmatrix}=N^2i-Ni^2=Ni(N-i),
\end{align*}
and since $\det(\bm{A}_N\bm{Z})=N$, we have $v_i=i(N-i)$, as required.  If instead $p\ne1/2$, then we have $\delta_i=1/(1-2p)$, and thus
\begin{align*}
\det(\bm{A}_i^{\star}\bm{Z}^{\star})=\begin{vmatrix}
\sum_{j=1}^Nz_1^j & \sum_{j=1}^N\delta_j\\
\sum_{j=1}^{i\vphantom{k^{k^k}}}z_1^j & \sum_{j=1}^{i\vphantom{k^{k^k}}}\delta_j
\end{vmatrix}=\frac{i\sum_{j=1}^Nz_1^j-N\sum_{j=1}^iz_1^j}{1-2p}.
\end{align*}
Then we have
\begin{align*}\SADS
v_i=\frac{i\sum_{j=1}^Nz_1^j-N\sum_{j=1}^iz_1^j}{(1-2p)\sum_{j=1}^Nz_1^j}=\frac{1}{1-2p}\left(i-\frac{N\sum_{j=1}^iz_1^j}{\sum_{j=1}^Nz_1^j}\right),
\end{align*}
and the required result follows immediately from the fact that
\[
\frac{\sum_{j=1}^iz_1^j}{\sum_{j=1}^Nz_1^j}=\frac{1-z_1^i}{1-z_1^N}
\]
for
$p\ne1/2$.
\end{proof}

\section{Recurrence/Transience and Stationary Distributions}
\label{sec:rtsd}

For the simple random leap with absorbing barriers, it may be trivially noted that the absorbing states are positive recurrent while the other states are transient.
So we return to the
non-absorbing random walks and leaps on the infinite state space~$\mathbb{Z}$.
This state space
provides a more interesting setting for the consideration of recurrence or transience, as well as stationary distributions or the lack thereof.  To begin, recall
the following standard result on simple random walks.

\begin{lem}[e.g., \citeauthor{feller1968}, \citeyear{feller1968}; \citeauthor{karlin1998}, \citeyear{karlin1998}]
\label{lem:walk-rtsd}
The simple random walk is transient if $p\ne1/2$ (equivalently $\mu\ne0$) and null recurrent if $p=1/2$ (equivalently $\mu=0$).  In particular, the simple random walk has no stationary distribution.
\end{lem}

We now prove that the
results of Lemma~\ref{lem:walk-rtsd} carry over to the general case of random leaps as well.

\begin{thm}\label{thm:rtsd}
The simple random leap is transient if $\mu\ne0$ and null recurrent if $\mu=0$.  In particular, the simple random leap has no stationary distribution.
\end{thm}

\begin{proof}
We first consider the transience or recurrence the simple random leap.  To do so, it
suffices to show whether state~$0$ is recurrent or transient, so we condition on $X_0=0$ throughout the proof.  First suppose $\mu\ne0$.  Then by the strong law of large numbers, $n^{-1}X_n\to\mu$ a.s., and hence $X_n\to\pm\infty$ a.s.\ according to the sign of $\mu$.  Now observe that state~$0$ is recurrent if and only if there exists a.s.\ a subsequence~$X_{m_n}$ of~$X_n$ such that $X_{m_n}=0$ for every~$n\ge0$.  Since $X_n\to\pm\infty$ a.s., such a subsequence a.s.\ does not exist.  Therefore state~$0$ is transient.

Instead suppose $\mu=0$.  In this case, we suppose that state~$0$ is transient and show that this leads to a contradiction.  Define $a=\sum_{n=0}^\infty P(X_n=0)$, and note that $a<\infty$ since state~$0$ is transient.  Also, for any $i\in\mathbb{Z}$, define $T_i=\min\{n\ge0:X_n=i\}$, and observe that
\begin{align*}
\sum_{n=0}^\infty P(X_n=i)
&=
\sum_{n=0}^\infty\sum_{t=0}^\infty P\left(X_n=i\Given T_i=t\right)\;P(T_i=t)\\
&=
\sum_{n=0}^\infty\sum_{t=0}^n P(X_{n-t}=0)\;P(T_i=t)\\
&=
\sum_{t=0}^\infty\sum_{n=t}^\infty P(X_{n-t}=0)\;P(T_i=t)
=
\left[\sum_{t=0}^\infty P(T_i=t)\right]\left[\sum_{n=0}^\infty P(X_{n}=0)\right]
\le a.
\end{align*}
Hence, $\sum_{n=0}^\infty P(X_n=i)\le a$ for all~$i\in\mathbb{Z}$.
Now define $M=\lceil\max\{100\, a^2\sigma^2,\;1/\sigma^2\}\rceil$ where $\lceil\,\cdot\,\rceil$ denotes the ceiling function, and let $w=\lfloor{\sqrt{2M\sigma^2}}\rfloor$, where $\lfloor\,\cdot\,\rfloor$ denotes the floor function.  Note that $w\ge1$.
Next, observe that $E(X_n)=0$ and $\text{Var}(X_n)=n\sigma^2\le M\sigma^2$ for every $n\in\{1,\ldots,M\}$.  Then by Chebyshev's inequality,
\bas
P\left(|X_n|\le\sqrt{2M\sigma^2}\right)\ge\frac{1}{2}
\eas
for every $n\in\{0,\ldots,M\}$.
Then
\begin{align*}
\frac{M}{2}
\le\sum_{n=1}^M P\left(|X_n|\le\sqrt{2M\sigma^2}\right)
&\le\sum_{n=0}^\infty P\left(|X_n|\le\sqrt{2M\sigma^2}\right)\\
&=\sum_{n=0}^\infty\sum_{i=\mathrlap{-w}\;}^{w}P\left(X_n=i\right)\\
&=\sum_{\;\mathllap{i=}-w}^{w}\sum_{n=0}^\infty P\left(X_n=i\right)
\le(2w+1)a\le3wa\le3a\sqrt{2M\sigma^2},
\end{align*}
from which it follows that $\sqrt{M}\le6a\sqrt{2\sigma^2}$, and thus
$M\le 72\,a^2\sigma^2$, which contradicts the definition of $M$.  Therefore state~$0$ is recurrent.

Finally, we show that the simple random leap has no stationary distribution, which also establishes null recurrence when $\mu=0$.  Suppose to the contrary that $\bm{\pi}$ is a stationary distribution, and let $g=\gcd\{j:p_j+q_j>0\}$, where $\gcd$ denotes the greatest common divisor of a set of integers.  If $g>1$ then we may partition the state space~$\mathbb{Z}$ into $g$ communicating classes and apply the argument below to each class individually.  Thus, we may assume without loss of generality that $g=1$, in which case the chain is irreducible.
For every $j\in\mathbb{Z}$, define the distribution~$\bm{\pi}^{(j)}$ according to $\pi^{(j)}_i=\pi_{i+j}$ for all $i\in\mathbb{Z}$, and note that $\bm{\pi}^{(j)}$ is also a stationary distribution by the chain's spatial homogeneity.
Since the chain is irreducible, there exists at most one stationary distribution, which implies that $\pi_{i+j}=\pi_i$ for all $i,j\in\mathbb{Z}$.  However, this is a contradiction since there is no uniform
probability
distribution on the infinite set~$\mathbb{Z}$.  Thus, no stationary distribution exists.
\end{proof}

\begin{rem}
It should be noted that
the transience/recurrence result of
Theorem~\ref{thm:rtsd} is a special case of a more general result.  In particular, \citet{chung1951} proved that the Markov chain formed by the sequence of sums of i.i.d.\ random vectors with finite covariance is recurrent if and only if the mean is zero and the dimension is two or less.
Their
result can also be regarded as an extension of the renowned result \citep{polya1921} that the simple random walk
on~$\mathbb{Z}^d$
is recurrent if and only if $d<3$.
\end{rem}

\section{Reflecting Random Leaps and Stationary \\Distributions}
\label{sec:reflect}

Section~\ref{sec:rtsd} demonstrated
that the simple random walk and leap have no stationary distribution.  For the absorbing random walks and leaps of Section~\ref{sec:leap}, the stationary distribution exists but is somewhat trivial.
However, there exist other variants of the
classical
random walk for which a nontrivial stationary distribution exists.
These variants feature a state space with either one or two endpoints that are reflecting, in the sense that the chain can move from either barrier back to the interior of the state space.  It is
thus natural to consider the extension of this idea to random leaps in order to investigate the existence and properties of similar stationary distributions
in the random leap setting.  Such reflecting random leaps may also be an interesting topic of study in their own right.
We study reflecting random leaps in the two distinct settings in which the state space is (i)~finite or (ii)~infinite.

\subsection{Finite State Space}
\label{subsec:reflect-finite}

The behavior of the reflecting random leap is conceptually straightforward:  It is a simple random leap modified so that at each step, any probability of moving to a state not in $\{0,\ldots,N\}$ is instead reassigned to the nearest barrier state, i.e., $0$ or $N$, as defined formally below.

\begin{defn}\label{defn:reflect-finite}
Let $\{X_n^\star:n\ge0\}$ be a simple random leap, and let $N\ge k_p+k_q$ be an integer.  Now define a Markov chain $\{X_n:n\ge0\}$ with state space $\{0,\ldots,N\}$ and transition probabilities
\bas\SADS
P\left(X_{n+1}=j\Given X_n=i\right)=\begin{cases}
P\left(X_{n+1}^\star\le 0\Given X_n^\star=i\right) & \text{ if }j=0,\\
P\left(X_{n+1}^\star=j\Given X_n^\star=i\right) & \text{ if }j\in\{1,\ldots,N-1\},\\
P\left(X_{n+1}^\star\ge N\Given X_n^\star=i\right) & \text{ if }j=N.\\
\end{cases}
\eas
We call the resulting Markov chain a \emph{reflecting random leap}.  (We may also attach the modifier \emph{two-sided} to clearly distinguish it from the one-sided reflecting random leap defined in Definition~\ref{defn:reflect-infinite}
below).
\end{defn}

Note that the reflecting random leap as defined in Definition~\ref{defn:reflect-finite} does not ``bounce back'' when it would otherwise breach the barrier (as its name might suggest).  Instead, it is absorbed by the barrier and may reflect back to the interior of the state space on the next step. In this sense, Definition~\ref{defn:reflect-finite} above is more analogous to the classical reflecting random walk.
We also make one additional assumption
(related to irreducibility)
when working with reflecting simple random leaps, which we will impose for the remainder of Section~\ref{sec:reflect}.

\begin{assum}\label{assum:gcd}
The greatest common divisor of the set $\{j:p_j+q_j>0\}$ is $1$.
\end{assum}

Since the reflecting random leap has a finite state space $\{0,\ldots,N\}$, it is clear that a stationary distribution exists.  Assumption~\ref{assum:gcd}
above
ensures that the chain is irreducible, which in turn guarantees that the stationary distribution is unique.
We now provide
two definitions.

\begin{defn}
The \emph{reverse characteristic polynomial} of a random leap is
\[
\psi(z)=\smashoperator{\sum_{j=0}^{2k-1}} \gamma_{j}\,z^{j},
\]
where the coefficients
$\gamma_{j}$ are given by
\[
\gamma_{j}=\begin{cases}
-\sum_{\ell=k-j}^k p_\ell & \text{ if }0\le j\le k-1,\\
\phantom{-}\sum_{\ell=j-k+1}^k q_\ell & \text{ if }k\le j\le 2k-1.
\end{cases}
\]
\label{defn:char-rev}
\end{defn}

The close connection between a random leap's characteristic and reverse characteristic polynomials can be seen by noting
that $\gamma_{j}=-c_{2k-j-1}$ for each $j\in\{0,\ldots,2k-1\}$, which implies that $\psi(z)=z^{2k-1}\,\chi(z^{-1})$ for all $z\ne0$.  Then the nonzero roots of $\psi(z)$ are precisely $z_1^{-1},\ldots,z_{s\vphantom1}^{-1}$, the multiplicative inverses of the nonzero roots of $\chi(z)$.
Let $y_1,\ldots,y_s$ denote the values $z_1^{-1},\ldots,z_{s\vphantom1}^{-1}$ sorted in ascending order of absolute value (with ties broken arbitrarily).  For each $j\in\{1,\ldots,s\}$, let~$\bm{Y}^{(j)}$ be the $(N+ r-1)\times  r_j$ matrix with elements
\begin{align}
Y^{(j)}_{i,\ell}=(i+1-k_p)^{\ell-1}y_j^{i+1-k_p},
\label{Y-elements}
\end{align}
with the convention that $0^0=1$.
Next, let $\bm{Y}$ be the $(N+ r-1)\times r$ matrix
\begin{align}
\bm{Y}=\begin{pmatrix}\bm{Y}^{(1)}&\bm{Y}^{(2)}&\cdots&\bm{Y}^{(s)}\end{pmatrix}.
\label{Y-matrix}
\end{align}
We now provide
the second of the aforementioned two definitions.

\begin{defn}
For each $i\in\{0,\ldots,N\}$ the $i$th \emph{modified accordion matrix} $\smash{\bm{A}_i^{\dagger}}$ is the \mbox{$ r\times(N+ r-1)$} matrix
\bas\SADS
\bm{A}_i^{\dagger}=\begin{cases}\begin{pmatrix}\bm{I}_{k_p-1}&\bm{0}&\bm{0}&\bm{0}\\ \bm{0}&\bm{1}_i^\T&\bm{0}_{N-i+1}^\T&\bm{0}\\ \bm{0}&\bm{0}_i^\T&\bm{1}_{N-i+1}^\T&\bm{0}\\ \bm{0}&\bm{0}&\bm{0}&\bm{I}_{k_q-2}\end{pmatrix}&\text{ if }k_q\ge2,\\ \vspace{-1em}\\
\begin{pmatrix}\bm{I}_{k_p-1}&\bm{0}&\bm{0}\\ \bm{0}&\bm{1}_i^\T&\bm{0}_{N-i}^\T&\end{pmatrix}&\text{ if }k_q=1.\end{cases}
\eas
\end{defn}
\begin{rem}
Note that the modified accordion matrix~$\smash{\bm{A}_i^{\dagger}}$ may be obtained from the accordion matrix~$\bm{A}_i$ by first reversing the roles of $k_p$ and $k_q$ in
Definition~\ref{defn:accordion}
and then adding $\begin{pmatrix}\bm{0}_{k_p+i-1}^\T&\bm{1}_{N-i}^\T&\bm{0}_{k_q-1}^\T\end{pmatrix}$ to the $(k_p+1)$st row if this row exists.
\end{rem}

Also, let $\smash{\bm{A}^\ddagger}$ be the matrix formed from $\smash{\bm{A}_i^\dagger}$ by first adding the $k_p$th row to the $(k_p+1)$st row (if this row exists), then replacing the $k_p$th row with $\bm{\eta}^\T$, noting that the resulting matrix no longer depends on~$i$.  

Now let $\bm{\eta}$ be the vector of length~$N+ r-1$ with elements
\begin{align}
\eta_j=\begin{cases}\sum_{\ell=j-k_p+1}^{k_q} (\ell-j+k_p)q_\ell&\text{ if }k_p\le j\le r,\\
0&\text{ otherwise}.\end{cases}
\label{eta}
\end{align}
Next, for each $i\in\{0,\ldots,N\}$, let $\bm{W}_i$ be the $( r+1)\times( r+1)$ matrix
\[
\bm{W}_i=\begin{pmatrix}\bm{A}_i^\dagger\bm{Y}&\bm{e}_{k_p-1}-\bm{e}_{k_p} \\ \hphantom{\bm{A}_i^\dagger}\mathllap{\bm{\eta}^\T}\bm{Y}&\sum_{\ell=1}^{k_p}p_\ell-\sum_{\ell=1}^{k_q}\ell q_\ell\end{pmatrix},
\]
where $\bm{e}_{k_p-1}$ and $\bm{e}_{k_p}$ denote the $(k_p-1)$st and $k_p$th unit vectors of length~$ r$, taking the former to be a zero vector when $k_p=1$.
Finally, we state one additional lemma.

\begin{lem}\label{lem:stationary}
Let $\bm{P}$ be the transition matrix for a recurrent and irreducible Markov chain with countable state space~$S$.
If
$\bm{m}=\{m_i:i\in S\}\ne\bm0$ is a
complex-valued vector
such that $\bm{m}\bm{P}=\bm{m}$ and 
$\sum_{i\in S}m_i$ is well-defined
with $|\sum_{i\in S}m_i|<\infty$, then
$\sum_{i\in S}m_i\ne0$, and
$(\sum_{i\in S}m_i)^{-1}\,\bm{m}$ is the unique stationary distribution of the Markov chain.
\end{lem}

\begin{proof}
The existence and uniqueness of a stationary measure~$\bm{\pi}^\star$ follow immediately from the recurrence and irreducibility of the Markov chain.
Now note that since the 
probabilities~$P_{i,j}$ in~$\bm{P}$ are real,
both $\Re(\bm{m})$ and $\Im(\bm{m})$ are stationary signed measures
on~$S$,
where $\Re(\cdot)$ and~$\Im(\cdot)$ denote, respectively, the real and imaginary parts (taken elementwise).
Note that at least one of these signed measures, say $\Re(\bm{m})$, is
not~$\bm0$ since $\bm m\ne\bm0$.
Then it suffices to show that either $\Re(\bm{m})$ or $-\Re(\bm{m})$ is a nonnegative measure
on~$S$.
(The proof for the imaginary part is in essence identical.)  We now prove this result by contradiction.  

Suppose neither $\Re(\bm{m})$ nor $-\Re(\bm{m})$ is a measure.
Then the sets
$S_+=\{i\in S:\Re(m_i)>0\}$ and $S_-=\{i\in S:\Re(m_i)<0\}$ are both nonempty.  Let $s_+\in S_+$ and $s_-\in S_-$.
Since the Markov chain is irreducible, there exists an integer
$n\ge1$
such that $P^n_{s_-,s_+}>0$,
where $P^n_{s_-,s_+}$ denotes the $(s_-,s_+)$th element of $\bm{P}^n$.
Then
\begin{align*}
\smashoperator{\sum_{i\in S_+}}\Re(m_i)
&=\sum_{\mathllap{i}\in S_+}\sum_{j\in \mathrlap{S_+}}\Re(m_j)\,P^n_{j,i}
+\sum_{\mathllap{i}\in S_+}\sum_{j\in \mathrlap{S_-}}\Re(m_j)\,P^n_{j,i}\\
&=\smashoperator{\sum_{j\in S_+}}\Re(m_j)\smashoperator{\sum_{i\in S_+}}P^n_{j,i}
+\smashoperator{\sum_{j\in S_-}}\Re(m_j)\smashoperator{\sum_{i\in S_+}}P^n_{j,i}\\
&\le\smashoperator{\sum_{j\in S_+}}\Re(m_j)+\Re(m_{s_-})\,
P^n_{s_-,s_+}
<\smashoperator{\sum_{j\in S_+}}\Re(m_j),
\end{align*}
a contradiction.
(Note that both terms in the expression on the right-hand side of the first line above are finite since $\sum_{i\in S}m_i$ is well-defined with $|\sum_{i\in S}m_i|<\infty$.
Note also that the interchange of the summation order in the second equality above is justified by Tonelli's theorem, since the summands of the first term are all nonnegative and the summands of the second term are all nonpositive.)
Thus, either $\Re(m_i)\ge0$ for all $i\in S$ or $\Re(m_i)\le0$ for all $i\in S$, and hence either $\Re(\bm{m})$ or $-\Re(\bm{m})$ is a measure.
\end{proof}

Then the following result provides the form of the stationary distribution.

\begin{thm}\label{thm:pi}
The stationary distribution $\bm{\pi}=(\pi_0\;\;\cdots\;\;\pi_N)^\T$ of the reflecting random leap is given by
\basn\SADS\label{pi-reflect}
\pi_i=\frac{\det(\bm{W}_i)}{\sum_{j=0}^N\det(\bm{W}_j)}
\easn
for each $i\in\{0,\ldots,N\}$.
\end{thm}

\begin{proof}
Note that the reflecting random leap is recurrent (due to its finite space) and irreducible (by Assumption~\ref{assum:gcd}) so Lemma~\ref{lem:stationary} holds.  Then it suffices to show that the vector $(\,\det(\bm{W}_0)\;\;\cdots\;\;\det(\bm{W}_N)\,)^\T$ is stationary.
The proof of this result consists of three main
parts. First,
we establish that any stationary vector satisfies a recurrence relation of order~$ r$ similar to the one used in the proof of Theorems~\ref{thm:u}~and~\ref{thm:v}.  Second, we find $ r$~boundary conditions that uniquely determine the values of the recurrence relation.  Third, we perform some matrix manipulations to express the solution in the desired form
in~(\ref{pi-reflect}).

We begin by establishing the aforementioned recurrence relation.
Define $d_i=\pi_i-\pi_{i-1}$ for each~$i\in\{1,\ldots,N\}$.
Then for each $i\in\{k_p,\ldots,N-k_q\}$,
\begin{align}
\pi_i=\sum_{j=0}^N \pi_j\,P\left(X_{n+1}=i\Given X_n=j\right)=\sum_{j=1}^{k_p} p_j\,\pi_{i-j}+\sum_{j=1}^{k_q} q_j\,\pi_{i+j}.
\label{pi-stat-middle}
\end{align}
Then we have
\begin{align*}\SADS
0&=\sum_{j=1}^{k_p} p_j(\pi_{i-j}-\pi_i)+\sum_{j=1}^{k_q} q_j(\pi_{i+j}-\pi_i)\notag\\
&=\sum_{j=1}^{k_q}\left(\sum_{\ell=j}^{k_q} q_\ell\right)d_{i+j}-\sum_{j=1}^{k_p}\left(\sum_{\ell=j}^{k_p} p_\ell\right)d_{i-j+1},
\end{align*}
which implies that
\begin{align}\SADS
d_i&=\smashoperator{\sum_{j=1}^{k_q-1}}\left(-\frac{1}{q_{k_q}}\sum_{\ell=j}^{k_q} q_\ell\right)d_{i-k_q+j}+\sum_{j=1}^{k_p}\left(\frac{1}{q_{k_q}}\sum_{\ell=j}^{
k_p
} p_\ell\right)d_{i-k_q-j+1}\notag\\
&=-\frac{1}{q_{k_q}}\sum_{j=1}^{ r}\gamma_{k+k_q-j-1}d_{i-j}
\label{pi-d-recurrence}
\end{align}
for each $i\in\{ r+1,\ldots,N\}$.  Thus, $d_{ r+1},\ldots,d_N$ are governed by 
a linear recurrence relation
of order $ r$ with characteristic polynomial
\begin{align*}
\psi^\star(z)=\frac{1}{q_{k_q}}\sum_{j=1}^{ r+1}\gamma_{k+k_q-j}z^{ r+1-j}
=\frac{z^{k_p-k}}{q_{k_q}}\sum_{j=k\mathrlap{-k_q+1}\,}^{k+\mathrlap{k_p}}\gamma_{2k-j}z^{2k-j}
=\frac{z^{k_p-k}}{q_{k_q}}\;\psi(z)
\end{align*}
since $c_{2k-j}=0$ if $j<k-k_q+1$ or $j>k+k_p$.  Thus, the $ r$ roots of $\psi^\star(z)$ coincide with the $ r$ nonzero roots of $\psi(z)$, which are $z_1^{-1},\ldots,z_s^{-1}$, or equivalently $y_1,\ldots,y_s$.

The general form of the solution to such a recurrence relation is
\begin{align}
d_i=\sum_{j=1}^s\sum_{\ell=1}^{ r_j}b^\star_{j,\ell}\,i^{\ell-1}
y_j^{i}
,
\label{pi-d-general}
\end{align}
with the convention that $0^0=1$,
and where the $b^\star_{j,\ell}$ are coefficients that must be found through boundary conditions, of which $ r$ are needed.  By treating $d_1,\ldots,d_ r$ as being determined by these boundary conditions, the above solution may be taken to hold for all $i\in\{1,\ldots,N\}$.

We now proceed to the second part of the proof, in which we derive these boundary conditions.
Suppose we permit the recurrence relation to be extended in each direction to also define for convenience the quantities $d_{2-k_p},\ldots,d_0$ and $d_{N+1},\ldots,d_{N+k_q-1}$.
Then for each $i\in\{N-k_q+1,\ldots,N-1\}$,
\begin{align}
\pi_i=\sum_{j=0}^N \pi_j\,P\left(X_{n+1}=i\Given X_n=j\right)=\sum_{j=1}^{k_p} p_j\,\pi_{i-j}+\smashoperator{\sum_{j=1}^{N-i}} q_j\,\pi_{i+j}.
\label{pi-stat-upper}
\end{align}
Thus, we have
\begin{align*}
0&=\sum_{j=1}^{k_p} p_j(\pi_{i-j}-\pi_i)+\smashoperator{\sum_{j=1}^{N-i}} q_j(\pi_{i+j}-\pi_i)-\smashoperator{\sum_{j=N-i+1}^{k_q}}q_j\pi_i\\
&=\sum_{j=1}^{N-i}\left(\sum_{\ell=j}^{N-i} q_\ell\right)d_{i+j}-\sum_{j=1}^{k_p}\left(\sum_{\ell=j}^{k_p} p_\ell\right)d_{i-j+1}-\left(\pi_0+\sum_{j=1}^i d_j\right)\sum_{\ell=N\mathrlap{-i+1}}^{k_q} q_\ell\\
&=\sum_{j=1}^{k_q}\left(\sum_{\ell=j}^{k_q} q_\ell\right)d_{i+j}-\sum_{j=1}^{k_p}\left(\sum_{\ell=j}^{k_p} p_\ell\right)d_{i-j+1}
-\sum_{\mathllap{j=N}-i+\mathrlap{1}}^{k_q}\left(\sum_{\ell=j}^{k_q} q_\ell\right)d_{i+j}-\left(\pi_0+\sum_{j=1}^{N} d_j\right)\sum_{\ell=N\mathrlap{-i+1}}^{k_q} q_\ell\\
&=-\sum_{\mathllap{j=N}-i+\mathrlap{1}}^{k_q}\left(\sum_{\ell=j}^{k_q} q_\ell\right)d_{i+j}-\left(\pi_0+\sum_{j=1}^{N} d_j\right)\sum_{\ell=N\mathrlap{-i+1}}^{k_q} q_\ell
\end{align*}
for each $i\in\{N-k_q+1,\ldots,N-1\}$,
where the last equality holds because $d_{i+k_q}$ satisfies the recurrence relation for each $i\in\{N-k_q+1,\ldots,N-1\}$.  This implies that
\[
\sum_{\mathllap{j=N}-i+\mathrlap{1}}^{k_q}\left(\sum_{\ell=j}^{k_q} q_\ell\right)d_{i+j}=-\left(\pi_0+\sum_{j=1}^{N} d_j\right)\sum_{\ell=N\mathrlap{-i+1}}^{k_q} q_\ell
\]
for each $i\in\{N-k_q+1,\ldots,N-1\}$, which may be represented in matrix form as
\[
\begin{pmatrix}
\sum_{j=k_q}^{k_q}q_j&\cdots&0\\
\vdots&\ddots&\vdots\\
\sum_{j=2}^{k_q}q_j&\cdots&\sum_{j=k_q}^{k_q}q_j
\end{pmatrix}
\begin{pmatrix}d_{N+1}\\ \vdots\\d_{N+k_q-1}\end{pmatrix}
=-\left(\pi_0+\sum_{j=1}^{N} d_j\right)
\begin{pmatrix}\sum_{j=k_q}^{k_q}q_j\\ \vdots\\ \sum_{j=2}^{k_q}q_j\end{pmatrix}.
\]
Gaussian elimination reduces this to
\begin{align}\label{pi-bd-upper}
d_{N+1}=-\left(\pi_0+\sum_{j=1}^{N} d_j\right), \qquad d_i=0 \text{ for each }i\in\{N+2,\ldots,N+k_q-1\}.
\end{align}
Similarly, for each $i\in\{1,\ldots,k_p-1\}$,
\begin{align}
\pi_i=\sum_{j=0}^N \pi_j\,P\left(X_{n+1}=i\Given X_n=j\right)=\sum_{j=1}^{i} p_j\,\pi_{i-j}+\sum_{j=1}^{k_q} q_j\,\pi_{i+j}.
\label{pi-stat-lower}
\end{align}
Thus, we have
\begin{align*}
0&=\sum_{j=1}^{i} p_j(\pi_{i-j}-\pi_i)+\sum_{j=1}^{k_q} q_j(\pi_{i+j}-\pi_i)-\sum_{j=i+1}^{k_p}p_j\pi_i\\
&=\sum_{j=1}^{k_q}\left(\sum_{\ell=j}^{k_q} q_\ell\right)d_{i+j}-\sum_{j=1}^{i}\left(\sum_{\ell=j}^{i} p_\ell\right)d_{i-j+1}-\left(\pi_0+\sum_{j=1}^i d_j\right)\sum_{\ell=i\mathrlap{+1}}^{k_p} p_\ell\\
&=\sum_{j=1}^{k_q}\left(\sum_{\ell=j}^{k_q} q_\ell\right)d_{i+j}-\sum_{j=1}^{k_p}\left(\sum_{\ell=j}^{k_p} p_\ell\right)d_{i-j+1}
+\sum_{j=i+1}^{k_p}\left(\sum_{\ell=j}^{k_p} p_\ell\right)d_{i-j+1}
-\pi_0\sum_{\ell=i\mathrlap{+1}}^{k_p} p_\ell\\
&=\sum_{\mathllap{j}=i+\mathrlap{1}}^{k_p}\left(\sum_{\ell=j}^{k_p} p_\ell\right)d_{i-j+1}
-\pi_0\sum_{\ell=i\mathrlap{+1}}^{k_p} p_\ell
\end{align*}
for each $i\in\{1,\ldots,k_p-1\}$,
where the last equality holds because $d_{i+k_q}$ satisfies the recurrence relation for each $i\in\{1,\ldots,k_p-1\}$.  This implies that
\[
\sum_{\mathllap{j}=i+\mathrlap{1}}^{k_p}\left(\sum_{\ell=j}^{k_p} p_\ell\right)d_{i-j+1}
=\pi_0\sum_{\ell=i\mathrlap{+1}}^{k_p} p_\ell
\]
for each $i\in\{1,\ldots,k_p-1\}$, which may be represented in matrix form as
\[
\begin{pmatrix}
\sum_{j=k_p}^{k_p}p_j&\cdots&\sum_{j=2}^{k_p}p_{j}\\
\vdots&\ddots&\vdots\\
0&\cdots&\sum_{j=k_p}^{k_p}p_j
\end{pmatrix}
\begin{pmatrix}d_{2-k_p}\\ \vdots\\d_0\end{pmatrix}
=\pi_0\begin{pmatrix}\sum_{j=2}^{k_p}p_j\\ \vdots\\ \sum_{j=k_p}^{k_p}p_j\end{pmatrix}.
\]
Gaussian elimination reduces this to
\begin{align}\label{pi-bd-lower}
d_i=0\text{ for each }i\in\{2-k_p,\ldots,-1\},\qquad d_0=\pi_0.
\end{align}
Now observe that
\begin{align}\SADS
\pi_0=\sum_{j=0}^N \pi_j\,P\left(X_{n+1}=i\Given X_n=j\right)=\pi_{0}\sum_{\ell=1}^{k_q}q_j+\sum_{j=1}^{k_q}\left(\sum_{\ell=j}^{k_q}q_\ell\right)\pi_{j},
\label{pi-stat-zero}
\end{align}
which implies that
\begin{align}\SADS
\sum_{j=1}^{k_q}\left[\sum_{\ell=j}^{k_q}(\ell-j+1)q_\ell\right]d_j=\pi_{0}\left(\sum_{\ell=1}^{k_p}p_\ell-\sum_{\ell=1}^{k_q}\ell\,q_\ell\right).
\label{pi-bd-zero}
\end{align}
(Note that the equation $\bm{\pi}=\bm{\pi}\bm{P}$ actually implies only $N$ linearly independent constraints, not $N+1$, so equating $\pi_N$ with the corresponding component of $\bm{\pi}\bm{P}$ would be superfluous.)
Then (\ref{pi-bd-upper}), (\ref{pi-bd-lower}), and (\ref{pi-bd-zero}) together provide the $ r$~boundary conditions needed to solve the recurrence relation.

Finally, we move to the last part of the proof, in which we manipulate the solution into the desired form.
Now define the length-$(N+ r-1)$ vector
$\bm{d}=(d_{2-k_p},\ldots,d_{N+k_q-1})$ and the length-$ r$ vector
$\bm{b}=(b_1,\ldots,b_ r)=(b^\star_{1,1},\ldots,b^\star_{s,r_s})$,
so that we may then write the general solution~(\ref{pi-d-general}) as simply $\bm{d}=\bm{Y}\bm{b}$.
Also, let $\bm{\phi}$ be the length-$ r$ vector with elements
\[
\phi_i=\begin{cases}
k_p-i & \text{ if } |i-k_p|=1,\\
\sum_{\ell=1}^{k_p}p_\ell-\sum_{\ell=1}^{k_q}\ell\,q_\ell & \text{ if }i=k_p,\\
0 & \text{ otherwise}.
\end{cases}
\]
Then (\ref{pi-bd-upper}), (\ref{pi-bd-lower}), and (\ref{pi-bd-zero}) may be combined as $\smash{\bm{A}^\ddagger\bm{d}=\pi_0\,\bm{\phi}}$,
hence
$\smash{\bm{A}^\ddagger\bm{Y}\bm{b}=\pi_0\,\bm{\phi}}$.
Now observe that the matrix $\bm{A}^\ddagger\bm{Y}$ cannot be singular.
(If it were, then there would be either no
solutions
or multiple solutions for $\pi_0^{-1}\bm{b}$, and thus
there would be
either no
solutions
or multiple solutions for $\{\pi_i/\pi_0:i\ge1\}$.
This would contradict
either the existence or uniqueness, respectively, of a stationary distribution.  Note also that $\pi_0>0$ since the Markov chain is irreducible.)
Then we may write
$\bm{b}=\pi_0(\bm{A}^\ddagger\bm{Y})^{-1}\bm{\phi}$ and apply Cramer's rule to write the elements $b_\ell$ of $\bm{b}$ as
\begin{align}\SADS
b_\ell=\frac{\pi_0}{\det(\bm{A}^\ddagger\bm{Y})}\begin{vmatrix}\bm{A}^\ddagger\bm{Y}&-\bm{\phi}\\ \bm{e}_\ell^\T&0\end{vmatrix},
\label{pi-cramer}
\end{align}
where $|\cdot|$ denotes the determinant and $\bm{e}_\ell$ is the $\ell$th unit vector of length $ r$.  Now let $Y_{j,\ell}$ denote the $(j,\ell)$th element of $\bm{Y}$, so that for each $i\in\{1,\ldots,N\}$,
\begin{align}
\pi_i
=\pi_0+\sum_{j=1}^i\sum_{\ell=1}^ r Y_{j+k_p-1,\,\ell}\,b_\ell&=
\pi_0+\frac{\pi_0}{\det(\bm{A}^\ddagger\bm{Y})}\begin{vmatrix}\bm{A}^\ddagger\bm{Y}&-\bm{\phi}\\ \sum_{j=k_p}^{k_p+i-1}\tilde{\bm{e}}_j^\T\bm{Y}&0\end{vmatrix}
\label{pi-manipulate}
\\
&=\frac{\pi_0}{\det(\bm{A}^\ddagger\bm{Y})}\begin{vmatrix}\bm{A}^\ddagger\bm{Y}&-\bm{\phi}\\ \sum_{j=k_p}^{k_p+i-1}\tilde{\bm{e}}_j^\T\bm{Y}&1\end{vmatrix}
=\pi_0\;\frac{\det(\bm{W}_i)}{\det(\bm{A}^\ddagger\bm{Y})},\notag
\end{align}
where $\tilde{\bm{e}}_j$ is the $j$th unit vector of length $N+r-1$.
Note that the last equality is obtained by switching the $k_p$th and last rows, negating the last column, and subtracting the $k_p$th row from the $(k_p+1)$st row unless the latter is also the last row.
Then setting $\sum_{j=0}^N\pi_j=1$ yields $\smash{\pi_0=\det(
\bm{A}^\ddagger
\bm{Y})/\sum_{j=0}^N\det(\bm{W}_j)}$, from which the result follows immediately.
\end{proof}

\subsection{Infinite State Space}
\label{subsec:reflect-infinite}

Another Markov chain of interest
is
the variant of the reflecting random leap that features a reflecting lower barrier at $0$ but lacks an upper barrier, which we formally define below.

\begin{defn}\label{defn:reflect-infinite}
Let $\{X_n^\star:n\ge0\}$ be a simple random leap.  Now define a Markov chain $\{X_n:n\ge0\}$ with state space $\{0,1,\ldots\}$ and transition probabilities
\[
P\left(X_{n+1}=j\Given X_n=i\right)=\begin{cases}
P\left(X_{n+1}^\star\le 0\Given X_n^\star=i\right) & \text{ if }j=0,\\
P\left(X_{n+1}^\star\mathrlap{{}=j}\hphantom{{}\le0}\Given X_n^\star=i\right) & \text{ if }j\ge1.
\end{cases}
\]
We call the resulting Markov chain a \emph{one-sided reflecting random leap}.
\end{defn}

Such a Markov chain may be informally conceptualized in some sense as the limit of a two-sided reflecting random leap on $\{0,\ldots,N\}$ as $N\to\infty$. It might therefore be hoped that the limits (if they exist) of the probabilities comprising the stationary distribution for the two-sided version might equal the corresponding probabilities for the one-sided version.  This is indeed the case as long as the convergence is uniform,
an idea which is further discussed in the Appendix.  However, we now derive the the stationary distribution for the one-sided form using different, though related, methods.
First,
we must prove a technical lemma about the roots of the polynomial $\psi(z)$.

\begin{lem}\label{lem:roots-loc}
Counting with multiplicity, the polynomial $\psi(z)$ has exactly $k_p-I(\mu\ge0)$ nonzero roots with absolute value strictly less than~$1$, exactly $k_q-I(\mu\le0)$ roots with absolute value strictly greater than~$1$, and exactly $I(\mu=0)$ roots with absolute value equal to~$1$.
\end{lem}

\begin{proof}
It suffices to show that the polynomial $\chi^\star(z)$ as defined in~(\ref{chi-star}) has exactly $k_p-I(\mu\ge0)$ roots with absolute value strictly greater than~$1$, exactly $k_q-I(\mu\le0)$ roots with absolute value strictly less than~$1$, and exactly $I(\mu=0)$ roots with absolute value equal to~$1$.

We first establish the last part of the statement.
By Descartes' rule of signs, $\chi^\star(z)$ has exactly one positive real root, and since $\chi^\star(1)=\mu$, it follows that $z=1$ is a (single) root of $\chi^\star(z)$ if $\mu=0$, and $z=1$ is not a root of $\chi^\star(z)$ if $\mu\ne0$.
Now define for $z\ne0$ the polynomial
\[
\tilde{\chi}(z)=(z-1)\,\chi^\star(z)=z^{k_q}\left(-1+\sum_{j=1}^{k_p}p_j z^{j}+\sum_{j=1}^{k_q}q_j z^{-j}\right)=\sum_{j=0}^{ r+1}\tilde{c}_{ r-j+1}z^{ r-j+1},
\]
with
\[
\tilde{c}_{ r-j+1}=\begin{cases}p_{k_p-j}&\text{ if }0\le j\le k_p-1,\\-1&\text{ if }j=k_p,\\q_{j-k_p}&\text{ if }k_p+1\le j\le  r+1.\end{cases}
\]
Note that the roots of $\tilde{\chi}(z)$ are the roots of $\chi^\star(z)$ plus an additional root at $z=1$.  For any $z\ne1$ with $|z|=1$, it is clear that $\tilde{\chi}(z)=0$ if and only if $z^j=1$ for every~$j$ such that $p_j+q_j>0$, noting that $\smash{\sum_{j=1}^{k_p}p_j+\sum_{j=1}^{k_q}q_j=1}$.
This condition is satisfied only if $z$ is a $g$th root of unity, where $g=\gcd\{j:p_j+q_j>0\}$.  But $g=1$ by Assumption~\ref{assum:gcd}, and hence the condition cannot be satisfied by any $z\ne1$.  Therefore $\chi^\star(z)$ has exactly $I(\mu=0)$ roots with absolute value equal to~$1$.

We now establish the first two parts of the statement in the case that $\mu\ne0$.  Observe that since $(1+\varepsilon)^j=1+j\varepsilon+O(\varepsilon^2)$ for any~$j\in
\mathbb{Z}
$, there exists $\varepsilon^\star>0$ such that for every~$\varepsilon$ with $|\varepsilon|<\varepsilon^\star$,
\bas\SADS
\max_{\mathllap{-}k_q\le j\le k_p}\left[(1+\varepsilon)^j-(1+j\varepsilon)\right]<\frac{|\varepsilon\mu|}{2}.
\eas
Then let $\varepsilon\in(0,\varepsilon^\star)$, and consider the $\mu<0$ and $\mu>0$ cases separately.

Suppose $\mu<0$.  Observe that
\begin{align*}
&\smashoperator{\sum_{j=0}^{k_p-1}}\,\left|\tilde{c}_{ r+1-j}\right|(1+\varepsilon)^{ r+1-j}
+\smashoperator{\sum_{j=k_p+1}^{ r+1}}\,\left|\tilde{c}_{ r+1-j}\right|(1+\varepsilon)^{ r+1-j}\\
&\qquad=(1+\varepsilon)^{k_q}\left(\sum_{j=1}^{k_p}p_{j} (1+\varepsilon)^{j}+\sum_{j=1}^{k_q}q_j (1+\varepsilon)^{-j}\right)\\
&\qquad<(1+\varepsilon)^{k_q}\left[\sum_{j=1}^{k_p}p_{j} \left(1+j\varepsilon-\frac{\varepsilon\mu}{2}\right)+\sum_{j=1}^{k_q}q_j \left(1-j\varepsilon-\frac{\varepsilon\mu}{2}\right)\right]\\
&\qquad=(1+\varepsilon)^{k_q}\left(1+\varepsilon\mu-\frac{\varepsilon\mu}{2}\right)
=(1+\varepsilon)^{k_q}\left(1+\frac{\varepsilon\mu}{2}\right)
<(1+\varepsilon)^{k_q}=\left|\tilde{c}_{k_q}\right|(1+\varepsilon)^{k_q}.
\end{align*}
Then by Rouch\'{e}'s theorem, $\tilde{\chi}(z)$ has exactly $k_q$ zeros with absolute value strictly less than $1+\varepsilon$.  Since $\varepsilon$ was arbitrarily small, it follows that $\tilde{\chi}(z)$ has exactly $k_q$ zeros with absolute value less than or equal to~$1$, which in turn implies that $\chi^\star(z)$ has exactly $k_q-1$ roots with absolute value less than or equal to~$1$.  Since $\mu\ne0$,\; $\chi^\star(z)$ has no roots with absolute value equal to~$1$, and thus it has $k_q-1$ roots with absolute value strictly less than~$1$.  Then the remaining $k_p$ roots of $\chi^\star(z)$ have absolute value strictly greater than~$1$.

Now suppose $\mu>0$.  Observe that
\begin{align*}
&\smashoperator{\sum_{j=0}^{k_p-1}}\,\left|\tilde{c}_{ r+1-j}\right|(1-\varepsilon)^{ r+1-j}
+\smashoperator{\sum_{j=k_p+1}^{ r+1}}\,\left|\tilde{c}_{ r+1-j}\right|(1-\varepsilon)^{ r+1-j}\\
&\qquad=(1-\varepsilon)^{k_q}\left(\sum_{j=1}^{k_p}p_{j} (1-\varepsilon)^{j}+\sum_{j=1}^{k_q}q_j (1-\varepsilon)^{-j}\right)\\
&\qquad<(1-\varepsilon)^{k_q}\left[\sum_{j=1}^{k_p}p_{j} \left(1-j\varepsilon+\frac{\varepsilon\mu}{2}\right)+\sum_{j=1}^{k_q}q_j \left(1-j\varepsilon+\frac{\varepsilon\mu}{2}\right)\right]\\
&\qquad=(1-\varepsilon)^{k_q}\left(1-\varepsilon\mu+\frac{\varepsilon\mu}{2}\right)
=(1-\varepsilon)^{k_q}\left(1-\frac{\varepsilon\mu}{2}\right)
<(1-\varepsilon)^{k_q}=\left|\tilde{c}_{k_q}\right|(1-\varepsilon)^{k_q}.
\end{align*}
Then by Rouch\'{e}'s theorem, $\tilde{\chi}(z)$ has exactly $k_q$ zeros with absolute value strictly less than $1-\varepsilon$.  Since $\varepsilon$ was arbitrarily small, it follows that $\tilde{\chi}(z)$ has exactly $k_q$ zeros with absolute value strictly less than~$1$, which in turn implies that $\chi^\star(z)$ has exactly $k_q$ roots with absolute value strictly less than~$1$.  
Since $\mu\ne0$,\; $\chi^\star(z)$ has no roots with absolute value equal to~$1$, and thus the remaining $k_p-1$ roots of $\chi^\star(z)$ have absolute value strictly greater than~$1$.

Finally, suppose $\mu=0$.  For each $\varepsilon$ satisfying $|\varepsilon|<\min\{p_{k_p},q_{k_q},1-p_{k_p},1-q_{k_p}\}$, let $\zeta_1(\varepsilon),\ldots,\zeta_ r(\varepsilon)$ be the roots of the characteristic equation $\chi^\star_\varepsilon(z)$ formed by replacing $p_{k_p}$ with $p_{k_p}+\varepsilon$ and replacing $q_{k_q}$ with $q_{k_q}-\varepsilon$ in the definition of $\chi^\star(z)$.  We may assume without loss of generality that $\zeta_1(\varepsilon)$ is the positive real root and that the other roots are ordered in such a way that $\zeta_j(\varepsilon)$ depends continuously on~$\varepsilon$ for each $j\in\{1,\ldots, r\}$.
(Assigning index~$1$ to the positive real root does not interfere with this continuity. Since the positive real root has multiplicity~$1$, it cannot leave the real line as $\varepsilon$ varies, as this would create a complex root without its complex conjugate as a root.)
Note that if $\varepsilon<0$, then $\#\{j:|\zeta_j(\varepsilon)|<1\}=k_q-1$ and $\#\{j:|\zeta_j(\varepsilon)|>1\}=k_p$.  However, if $\varepsilon>0$, then $\#\{j:|\zeta_j(\varepsilon)|<1\}=k_q$ and $\#\{j:|\zeta_j(\varepsilon)|>1\}=k_p-1$.  Now recall that any $z\ne1$ with $|z|=1$ cannot be a root of $\chi^\star(z)$.  Then for each $j\in\{2,\ldots, r\}$,\; $|\zeta_j(\varepsilon)|\ne1$ for any $\varepsilon$.  Since each $\zeta_j(\varepsilon)$ depends continuously on~$\varepsilon$, we have that for each~$j\in\{2,\ldots, r\}$, either $|\zeta_j(\varepsilon)|<1$ for all $\varepsilon$ or $|\zeta_j(\varepsilon)|>1$ for all $\varepsilon$.  Then it follows that $\#\{j\ge2:|\zeta_j(\varepsilon)|<1\text{ for all }\varepsilon\}=k_q-1$ and $\#\{j\ge2:|\zeta_j(\varepsilon)|>1\text{ for all }\varepsilon\}=k_p-1$, with $\zeta_1(\varepsilon)>0$ for all $\varepsilon<0$ and $\zeta_1(\varepsilon)<0$ for all $\varepsilon>0$.  Again, since $\zeta_1(\varepsilon)$ depends continuously on $\varepsilon$, it follows that $\zeta_1(0)=1$, which establishes the result.
\end{proof}

We now introduce some additional notation to permit a concise statement of the next result.
Recall the definition of $\bm{Y}$ in~(\ref{Y-matrix}).
Let $\bm{\Psi}$ be the matrix with
infinitely many
rows and $k_p$ columns formed by
first extending the matrix~$\bm{Y}$ downward (i.e., as $N\to\infty$) by allowing~(\ref{Y-elements}) to hold for every $i\ge1$, and then retaining only the first~$k_p$ columns of the resulting matrix.
We now provide another definition.
\begin{defn}
For each $i\in\{0,\ldots,N\}$ the $i$th \emph{half-accordion matrix}~$\bm{H}_i$ is the matrix with $k_p$ rows and infinitely many columns given by
\[
\bm{H}_i=\begin{pmatrix}\bm{I}_{k_p-1}&\bm{0}&\bm{0}&\cdots\\ \bm{0}&\bm{1}_i^\T&\bm{0}&\cdots\end{pmatrix}.
\]
\end{defn}
Next, let $\bm{\eta}_\infty$ be the vector of infinite length given by $\bm{\eta}_\infty^\T=(\bm{\eta}^\T\;\;\bm{0}\;\;\cdots)$,
where $\bm{\eta}$ is the same as in~(\ref{eta}).
Finally, for each $i\in\{0,1,\ldots\}$, let $\bm{\Omega}_i$ be the $(k_p+1)\times(k_p+1)$ matrix
\begin{align}
\bm{\Omega}_i=\begin{pmatrix}\hphantom{\bm{\eta}_\infty^\T}\mathllap{\bm{H}_i}\bm{\Psi}&\bm{e}_{k_p-1}-\bm{e}_{k_p}\\ \bm{\eta}_\infty^\T\bm{\Psi} & \sum_{\ell=1}^{k_p}p_\ell-\sum_{\ell=1}^{k_q}\ell q_\ell\end{pmatrix},
\end{align}
where $\bm{e}_{k_p-1}$ and $\bm{e}_{k_p}$ denote the $(k_p-1)$st and $k_p$th unit vectors of length~$k_p$, taking the former to be a zero vector when $k_p=1$.
Then the following result provides the form of the stationary distribution.

\begin{thm}\label{thm:pi-limit}
If $\mu\ge0$, then the one-sided reflecting random leap has no stationary distribution.  If $\mu<0$, then the one-sided reflecting random leap has a stationary distribution $\bm{\pi}=(\pi_0,\pi_1,\ldots)$ given by
\[
\pi_i=\frac{\det(\bm{\Omega}_i)}{\sum_{j=0}^\infty\det(\bm{\Omega}_j)}
\]
for each $i\in\{0,1,\ldots\}$.
\end{thm}

\begin{proof}
Let $\{X_n:n\ge0\}$ be a one-sided reflecting random leap.  We consider the $\mu\ge0$ and $\mu<0$ cases separately.

First suppose $\mu\ge0$, and suppose that $\bm{\pi}$ is a stationary distribution.  Since the chain is irreducible
by Assumption~\ref{assum:gcd},
we must have $\pi_i>0$ for every $i\ge0$.  Now let $X_0$ have distribution~$\bm{\pi}$.  Then $X_1$ also has distribution~$\bm{\pi}$.  Now note that $|X_1-X_0|\le k$, hence $E(X_1-X_0)$ exists and is finite.  Then
\begin{align*}\SADS
E(X_1-X_0)=E\left[E\left(X_1-X_0\Given X_0\right)\right]
&=\sum_{i=0}^\infty \pi_i\, E\left(X_1-X_0\Given X_0=i\right)\\
&=\smashoperator{\sum_{i=0}^{k_q-1}}\pi_i\left[\mu+\smashoperator{\sum_{\ell=i+1}^{k_q}}(\ell-i)q_\ell\right] +\smashoperator{\sum_{i=k_q}^\infty}\pi_i\,\mu\\
&\ge\left[\sum_{i=0}^{k_q\mathrlap{-1}\;}\pi_i\smashoperator{\sum_{\ell=i+1}^{k_q}}(\ell-i)q_\ell\right]
\ge
\pi_0\,k_q\,q_{k_q}>0.
\end{align*}
Thus, $X_1$ and $X_0$ do not have the same distribution, which is a contradiction since~$\bm{\pi}$ is stationary.  Therefore no stationary distribution exists.


Now suppose instead that $\mu<0$.
The proof for this case consists of three parts.  First, we show that the Markov chain is recurrent.
Second, we show that
\begin{align}
\det(\bm{\Omega}_i)=\sum_{j=0}^\infty \det(\bm{\Omega}_j)\,P\left(X_{n+1}=i\Given X_n=j\right)
\text{ for all }i\in\{0,1,\ldots\}.
\label{condition-stationary}
\end{align}
Third, we establish that
$|\sum_{i=0}^\infty\det(\bm{\Omega}_j)|<\infty$.
Then since the Markov chain is irreducible by Assumption~\ref{assum:gcd}, Lemma~\ref{lem:stationary} implies that
$\det(\bm{\Omega}_i)/\sum_{j=0}^\infty\det(\bm{\Omega}_j)$ is the unique stationary distribution.

First, we prove recurrence of the Markov chain $\{X_n:n\ge0\}$.  Since this chain is irreducible by Assumption~\ref{assum:gcd}, it suffices to show that at least one state is recurrent.
Note that we may express this Markov chain as
\begin{align*}
X_n=X_0+\sum_{i=1}^n\left(X_i-X_{i-1}\right)
=X_0^\star+\sum_{i=1}^n\left(X_i^\star-X_{i-1}^\star+\Delta_i\right)
=X_n^\star+\sum_{i=1}^n\Delta_i,
\end{align*}
where $\{X_n^\star:n\ge0\}$ is a corresponding simple (i.e., non-reflecting) random leap with $X_0^\star=X_0$ and where $\Delta_i$ is
the random variable
\[
\Delta_i=-(X_{i-1}+X_i^\star-X_{i-1}^\star)\;I(X_{i-1}+X_i^\star-X_{i-1}^\star<0).
\]
Observe that $0\le\Delta_i\le k_q\,I(X_{i-1}<k_q)$, which implies that
\begin{align}
0\le X_n-X_n^\star\le k_q\sum_{i=0}^{n-1} I(X_i<k_q).
\label{visits-bound}
\end{align}
Now note that $n^{-1}\sum_{i=1}^n(X_i^\star-X_{i-1}^\star)\to\mu<0$ almost surely as $n\to\infty$ by the strong law of large numbers.  Then $X_n^\star\to-\infty$ almost surely as $n\to\infty$.  However, $X_n\ge0$ for all $n\ge0$, and hence~(\ref{visits-bound}) implies that $\sum_{i=0}^\infty I(X_i<k_q)=\infty$ almost surely.  Thus, $\{X_n:n\ge0\}$ visits the finite set $\{0,\ldots,k_q-1\}$ infinitely often almost surely, from which it follows that at least one state in this finite set is recurrent.

Second, we show condition~(\ref{condition-stationary}).
For every $i\ge1$, define a vector~$\bm{\alpha}_{(i)}$ of length $i+k_q$ with elements
\[
\alpha_{(i),j}=\begin{cases}1&\text{ if }1\le j\le i-k_p,\\
\sum_{\ell=1}^{i-j}p_\ell+\sum_{\ell=1}^{k_q}q_\ell&\text{ if }i-k_p+1\le j\le i,\\
\sum_{\ell=j-i}^{k_q}q_\ell&\text{ if }i+1\le j\le i+k_q,\end{cases}
\]
taking any empty sum to be zero.  (Note that the $1\le j\le i-k_p$ case does not apply when $i\le k_p-1$.)  Now observe that we may rewrite these elements as
\begin{align*}
\alpha_{(i),j}
&=\begin{cases}1&\text{ if }1\le\ell\le i-k_p,\\
1+\gamma_{\ell-i+k-1}&\text{ if }i-k_p+1\le\ell\le i,\\
\gamma_{\ell-i+k-1}&\text{ if }i+1\le\ell\le i+k_q,\end{cases}
\end{align*}
with each $\gamma_j$ as given in Definition~\ref{defn:char-rev}.  Then for every~$i\ge1$,
\begin{align}
&\sum_{j=0}^\infty \det(\bm{\Omega}_j)\,P\left(X_{n+1}=i\Given X_n=j\right) \notag\\
&\qquad=\smashoperator{\sum_{j=1}^{\max\{i,k_p\}}} p_j\,\det(\bm{\Omega}_{i-j})+\sum_{j=1}^{k_q} q_j\,\det(\bm{\Omega}_{i+j}) \notag\\
&\qquad=\begin{vmatrix}
\begin{pmatrix}\bm{I}_{k_p-1}&\bm{0}&\bm{0}&\cdots\\ \bm{0}&\bm{\alpha}_{(i)}^\T&\bm{0}&\cdots\end{pmatrix}\bm{\Psi}&
\begin{matrix}
\tilde{\bm{e}}_{k_p-1}
\\-1+\sum_{j=i+1}^{k_p}p_j\end{matrix}
\\ \bm{\eta}_\infty^\T\bm{\Psi} & \sum_{\ell=1}^{k_p}p_\ell-\sum_{\ell=1}^{k_q}\ell q_\ell
\end{vmatrix},
\label{condition-stationary-intermediate}
\end{align}
where we write $\tilde{\bm{e}}_{k_p-1}$ to denote the $(k_p-1)$st unit vector of length~$k_p-1$, as opposed to length~$k_p$.
Next, for every $i\ge1$, define a vector
\[
\tilde{\bm{\alpha}}_{(i)}=\begin{cases}\bm{0}_{k_p-1}&\text{ if }i\ge k_p,\\
\begin{pmatrix}\bm{0}_{i-1}^\T&\gamma_{k-k_p}&\gamma_{k-k_p+1}&\cdots&\gamma_{k-i-2}&\gamma_{k-i-1}\end{pmatrix}^{
T}&\text{ if }i\le k_p-1.\end{cases}
\]
Now return to~(\ref{condition-stationary-intermediate}).  If $i\le k_p-1$, then add multiples of rows $i,\ldots,k_p-1$ to row~$k_p$ to obtain
\begin{align}\SADS
&\begin{vmatrix}
\begin{pmatrix}\bm{I}_{k_p-1}&\bm{0}&\bm{0}&\cdots\\ \bm{0}&\bm{\alpha}_{(i)}^\T&\bm{0}&\cdots\end{pmatrix}\bm{\Psi}&
\begin{matrix}
\tilde{\bm{e}}_{k_p-1}\\-1+\sum_{j=i+1}^{k_p}p_j\end{matrix}
\\ \bm{\eta}_\infty^\T\bm{\Psi} & \sum_{\ell=1}^{k_p}p_\ell-\sum_{\ell=1}^{k_q}\ell q_\ell
\end{vmatrix}\notag\\
&\qquad=
\begin{vmatrix}
\begin{pmatrix}\bm{I}_{k_p-1}&\bm{0}&\bm{0}&\cdots\\ \tilde{\bm{\alpha}}_{(i)}^\T&\bm{\alpha}_{(i)}^\T&\bm{0}&\cdots\end{pmatrix}\bm{\Psi}&
\begin{matrix}
\tilde{\bm{e}}_{k_p-1}\\-1\end{matrix}
\\ \bm{\eta}_\infty^\T\bm{\Psi} & \sum_{\ell=1}^{k_p}p_\ell-\sum_{\ell=1}^{k_q}\ell q_\ell
\end{vmatrix}.
\label{condition-stationary-almost}
\end{align}
If instead $i\ge k_p$, then~(\ref{condition-stationary-almost}) holds trivially.
Now observe that
\[
\begin{pmatrix}\tilde{\bm{\alpha}}_{(i)}^\T&\bm{\alpha}_{(i)}^\T\end{pmatrix}=\begin{pmatrix}\bm{1}_i^\T&\bm{0}_{k_q}^\T\end{pmatrix}+
\begin{pmatrix}\bm{0}_{i-k_p}^\T&\gamma_{k-k_p}&\gamma_{k-k_p+1}&\cdots&\gamma_{k+k_q-2}&\gamma_{k+k_q-1}\end{pmatrix}.
\]
Then since the columns of $\bm{\Psi}$ satisfy the recurrence relation corresponding to the reverse characteristic polynomial~$\psi(z)$ as defined in Definition~\ref{defn:char-rev}, we have
\[
\begin{pmatrix}\bm{I}_{k_p-1}&\bm{0}&\bm{0}&\cdots\\ \tilde{\bm{\alpha}}_{(i)}^\T&\bm{\alpha}_{(i)}^\T&\bm{0}&\cdots\end{pmatrix}\bm{\Psi}
=\begin{pmatrix}\bm{I}_{k_p-1}&\bm{0}&\bm{0}&\cdots\\ \bm{0}&\bm{1}_i^\T&\bm{0}&\cdots\end{pmatrix}\bm{\Psi}.
\]
Combining this with~(\ref{condition-stationary-intermediate})~and~(\ref{condition-stationary-almost}) immediately yields that condition~(\ref{condition-stationary}) holds for every $i\ge1$.
For $i=0$,
\begin{align}
&\sum_{j=0}^\infty \det(\bm{\Omega}_j)\,P\left(X_{n+1}=i\Given X_n=j\right)\notag\\
&\qquad=\sum_{j=1}^{k_q} q_j\,\det(\bm{\Omega}_0)
+\sum_{j=1}^{k_q}\sum_{\ell=j}^{k_q}q_\ell\,\det(\bm{\Omega}_{j}) \notag\\
&\qquad=\sum_{j=1}^{k_q} q_j
\begin{vmatrix}
\bm{H}_{0}\bm{\Psi}&\bm{e}_{k_p-1}-\bm{e}_{k_p}\\ \bm{\eta}_\infty^\T\bm{\Psi} & \sum_{\ell=1}^{k_p}p_\ell-\sum_{\ell=1}^{k_q}\ell q_\ell
\end{vmatrix}
+\sum_{j=1}^{k_q}\sum_{\ell=j}^{k_q}q_\ell
\begin{vmatrix}
\bm{H}_{j}\bm{\Psi}&\bm{e}_{k_p-1}-\bm{e}_{k_p}\\ \bm{\eta}_\infty^\T\bm{\Psi} & \sum_{\ell=1}^{k_p}p_\ell-\sum_{\ell=1}^{k_q}\ell q_\ell
\end{vmatrix}\notag\\
&\qquad=\begin{vmatrix}
\begin{pmatrix}\bm{I}_{k_p-1}&\bm{0}&\bm{0}&\cdots\\ \bm{0}&\bm{g}^\T&\bm{0}&\cdots\end{pmatrix}\bm{\Psi}& & &
\bm{e}_{k_p-1}-\sum_{\ell=1}^{k_q}(\ell+1)q_\ell\,\bm{e}_{k_p}
\\ \bm{\eta}_\infty^\T\bm{\Psi} & & & \sum_{\ell=1}^{k_p}p_\ell-\sum_{\ell=1}^{k_q}\ell q_\ell
\end{vmatrix},\label{condition-stationary-zero-intermediate}
\end{align}
where $\bm{g}$ is the length-$k_q$ vector with elements $g_j=\sum_{\ell=j}^{k_q}(\ell-j+1)q_\ell$.  Noting that 
$g_j=\eta_{j+k_p-1}$
for each $j\in\{1,\ldots,k_q\}$ and that
\[
\sum_{\ell=1}^{k_q}(\ell+1)q_\ell=1-\sum_{\ell=1}^{k_p}p_\ell+\sum_{\ell=1}^{k_q}\ell q_\ell,
\]
we may rewrite the determinant in~(\ref{condition-stationary-zero-intermediate}) as
\begin{align*}
&\begin{vmatrix}
\begin{pmatrix}\bm{I}_{k_p-1}&\bm{0}&\bm{0}&\cdots\\ \bm{0}&\bm{g}^\T&\bm{0}&\cdots\end{pmatrix}\bm{\Psi}& & &
\bm{e}_{k_p-1}-\sum_{\ell=1}^{k_q}(\ell+1)q_\ell\,\bm{e}_{k_p}
\\ \bm{\eta}_\infty^\T\bm{\Psi} & & & \sum_{\ell=1}^{k_p}p_\ell-\sum_{\ell=1}^{k_q}\ell q_\ell
\end{vmatrix}\\
&\qquad=
\begin{vmatrix}
\begin{pmatrix}\bm{I}_{k_p-1}&\bm{0}&\bm{0}&\cdots\end{pmatrix}\bm{\Psi}& & &
\tilde{\bm{e}}_{k_p-1}\\ 
\bm{\eta}_\infty^\T\bm{\Psi} & & & \sum_{\ell=1}^{k_p}p_\ell-\sum_{\ell=1}^{k_q}\ell q_\ell-1\\
\bm{\eta}_\infty^\T\bm{\Psi} & & & \sum_{\ell=1}^{k_p}p_\ell-\sum_{\ell=1}^{k_q}\ell q_\ell\phantom{{}-1}
\end{vmatrix}
=\det(\bm{\Omega}_0),
\end{align*}
where the last equality is obtained by subtracting the last row from the second-to-last row.  Thus, condition~(\ref{condition-stationary}) holds for~$i=0$, and hence for all~$i\ge0$.

Third and finally, we establish that $|\sum_{i=0}^\infty\det(\bm{\Omega}_j)|<\infty$.
Begin by noting that
\[
\sum_{j=0}^{2k-1}\gamma_j=\smashoperator{\sum_{j=k-k_p}^{k+k_q-1}}\gamma_j=\sum_{\ell=1}^{k_q}\ell q_\ell-\sum_{\ell=1}^{k_p}\ell p_\ell=-\mu>0.
\]
Now let $\breve{\bm{e}}_m$ denote the $m$th unit vector of infinite length.  Then for all~$i\ge 2r+2$,
\begin{align*}
-\mu\begin{pmatrix}\bm{0}_{k_p-1}^\T&\bm{1}^\T_i&\bm{0}&\cdots\end{pmatrix}\bm{\Psi}
&=\sum_{\mathllap{j=k}-k_p}^{\mathllap{k+}k_q-\mathrlap{1}}\sum_{m=\mathrlap{1}\vphantom{k_p}}^{i\vphantom{k_q}}\gamma_j\breve{\bm{e}}_{m+k_p-1}^\T\bm{\Psi}\\
&=\sum_{\mathllap{m}=1\vphantom{k_p}}^{\;\;\mathllap{i-k_p-k_q}+1}\sum_{j=k\mathrlap{-k_p}}^{k+\mathrlap{k_q-1}}\gamma_j\breve{\bm{e}}_{j+m-k+k_p}^\T\bm{\Psi}+\sum_{\mathllap{m}=1\vphantom{k_p}}^{k_q}\sum_{j=k\mathrlap{-k_p+m}}^{k+\mathrlap{k_q-1}}\gamma_j\breve{\bm{e}}_{m+k_p-1}^\T\bm{\Psi}\\
&\qquad-\sum_{\mathllap{m}=1\vphantom{k_p}}^{\mathllap{k_p}-1}\sum_{j=k\mathrlap{-k_p}}^{k-\mathrlap{k_p+m-1}}\gamma_j\breve{\bm{e}}_m^\T\bm{\Psi}
\quad+\quad\sum_{\;\mathllap{m=i-k_p-k_q}+2}^{i}\sum_{j=k\mathrlap{-k_p}}^{k+\mathrlap{k_q+m-i-2}}\gamma_j\breve{\bm{e}}_{m+k_p-1}^\T\bm{\Psi}.
\end{align*}
The first term of this expression is zero since the columns of $\bm{\Psi}$ satisfy the recurrence relation corresponding to $\psi(z)$.  Now let $\tilde{\bm{g}}$ be the vector of length $k_p-1$ with elements $\tilde{g}_j=\sum_{\ell=k_p-j+1}^{k_p}(\ell-k_p+j)p_\ell$, and let $\breve{\bm{g}}$ be the vector of length~$ r$ with elements
\[
\breve{\bm{g}}_j=\begin{cases}-\sum_{\ell=j}^{k_p}(\ell-j+1) p_\ell&\text{ if }1\le j\le k_p,\\ -\sum_{\ell=1}^{k_p}\ell p_\ell+\sum_{\ell=j-k_p}^{k_q}(\ell-j+k_p+1)q_\ell&\text{ if }k_p+1\le j\le  r.\end{cases}
\]
Then
\begin{align*}\SADS
-\mu\begin{pmatrix}\bm{0}_{k_p-1}^\T&\bm{1}^\T_i&\bm{0}&\cdots\end{pmatrix}\bm{\Psi}
=\begin{pmatrix}\tilde{\bm{g}}^\T&\bm{g}^\T&\bm{0}_{i-k_p-2k_q+1}^\T&\breve{\bm{g}}^\T&\bm{0}&\cdots\end{pmatrix}\bm{\Psi},
\end{align*}
and hence
\begin{align*}\SADS
-\mu\det(\bm{\Omega}_i)&=\begin{vmatrix}
\begin{pmatrix}\bm{I}_{k_p-1}&\bm{0}&\bm{0}&\cdots\end{pmatrix}\bm{\Psi}&
\tilde{\bm{e}}_{k_p-1}\\
\begin{pmatrix}\tilde{\bm{g}}^\T&\bm{g}^\T&\bm{0}_{i-k_p-2k_q+1}^\T&\breve{\bm{g}}^\T&\bm{0}&\cdots\end{pmatrix}\bm{\Psi}&\mu\\
\bm{\eta}_\infty^\T\bm{\Psi} & \sum_{\ell=1}^{k_p}p_\ell-\sum_{\ell=1}^{k_q}\ell q_\ell
\end{vmatrix}\\
&=\begin{vmatrix}
\begin{pmatrix}\bm{I}_{k_p-1}&\bm{0}&\bm{0}&\cdots\end{pmatrix}\bm{\Psi}&
\tilde{\bm{e}}_{k_p-1}\\
\begin{pmatrix}\tilde{\bm{g}}^\T&\bm{0}_{i-k_p-k_q+1}^\T&\breve{\bm{g}}^\T&\bm{0}&\cdots\end{pmatrix}\bm{\Psi}&\sum_{\ell=2}^{k_p}(\ell-1)p_\ell\\
\bm{\eta}_\infty^\T\bm{\Psi} & \sum_{\ell=1}^{k_p}p_\ell-\sum_{\ell=1}^{k_q}\ell q_\ell
\end{vmatrix}\\
&=\begin{vmatrix}
\begin{pmatrix}\bm{I}_{k_p-1}&\bm{0}&\bm{0}&\cdots\end{pmatrix}\bm{\Psi}&
\tilde{\bm{e}}_{k_p-1}\\
\begin{pmatrix}\bm{0}_{i-k_q}^\T&\breve{\bm{g}}^\T&\bm{0}&\cdots\end{pmatrix}\bm{\Psi}&0\\
\bm{\eta}_\infty^\T\bm{\Psi} & \sum_{\ell=1}^{k_p}p_\ell-\sum_{\ell=1}^{k_q}\ell q_\ell
\end{vmatrix}.
\end{align*}
Note that the $(k_p,j)$th element of the last matrix above is $\sum_{\ell=1}^{ r}\breve{g}_\ell\Psi_{\ell+i-k_q,\,j}$ if $1\le j\le k_p$ and $0$ if $j=k_p+1$.  Now let $C_{k_p,j}$ denote the cofactor of this element, and note that $C_{k_p,j}$ does not depend on~$i$.
Then
\bas\SADS
-\mu\det(\bm{\Omega}_i)=\sum_{j=1}^{k_p}C_{k_p,j}\sum_{\ell=1}^{ r}\breve{g}_\ell\Psi_{\ell+i-k_q,\,j},
\eas
and hence
\begin{align*}\SADS
-\mu\smashoperator{\sum_{i=2r+2}^{\infty}}\det(\bm{\Omega}_i)&=
\sum_{\mathllap{i=2}r+2}^{\infty}
\sum_{j=1}^{k_p}C_{k_p,j}\sum_{\ell=1}^{ r}\breve{g}_\ell\Psi_{\ell+i-k_q,\,j}\\
&=\sum_{j=1}^{k_p}C_{k_p,j}\sum_{\ell=1}^{ r}\breve{g}_\ell\smashoperator{\sum_{i=2 r+2}^{\infty}}\Psi_{\ell+i-k_q,\,j}.
\end{align*}
Note that the interchange of summation order above is permissible because the triple sum converges absolutely, which is shown below.
Now observe that by Lemma~\ref{lem:roots-loc}, the roots~$y_j$ that appear in~$
\bm{\Psi}$ all have absolute value less than~$1$.
Then
for each $j\in\{1,\ldots,k_p\}$, there exist a
geometrically distributed
random variable~$G_j$ and a nonnegative integer~$r_j$ such that $\sum_{m=k_p-1}^\infty\left|\Psi_{m,j}\right|=E(G_j^{r_j})$.  Define $M=\max_{1\le j\le k_p}E(G_j^{r_j})$, and note that $M<\infty$.  Then
\begin{align*}
\left|\sum_{i=0}^{\infty}\det(\bm{\Omega}_i)\right|
&\le\left|\sum_{i=0}^{2 r+\mathrlap{1}}\det(\bm{\Omega}_i)\right|+
\left|\sum_{i=2\mathrlap{r+2}}^\infty\det(\bm{\Omega}_i)\right|\\
&\le\left|\sum_{i=0}^{2 r+\mathrlap{1}}\det(\bm{\Omega}_i)\right|+
\frac{1}{|\mu|}\sum_{j=1}^{k_p}\,\left|C_{k_p,j}\right|\sum_{\ell=1}^{ r}\,\left|\breve{g}_\ell\right|\smashoperator{\sum_{i=2 r+2}^{\infty}}\,\left|\Psi_{\ell+i-k_q,\,j}\right|\\
&\le\left|\sum_{i=0}^{2 r+\mathrlap{1}}\det(\bm{\Omega}_i)\right|+
\frac{M}{|\mu|}\left(\sum_{j=1}^{k_p}\,\left|C_{k_p,j}\right|\right)\left(\sum_{\ell=1}^{ r}\,\left|\breve{g}_\ell\right|\right)<\infty.
\end{align*}
Thus,
$\det(\bm{\Omega}_i)/\sum_{j=0}^\infty\det(\bm{\Omega}_j)=\pi_i$
is the unique stationary distribution.
\end{proof}

\begin{rem}
An alternative approach to proving Theorem~\ref{thm:pi-limit} is discussed in the Appendix.  This approach relies upon the general technique of showing that if the probabilities associated with the stationary distribution for the finite state space $\{0,\ldots,N\}$ and converge uniformly to some limit as $N\to\infty$, then this limit is the stationary distribution for the infinite state space.  The Appendix rigorously elaborates on this idea.
\end{rem}

We now provide two examples to illustrate the generality of the above results.

\begin{ex}
Consider a one-sided reflecting random walk with $\mu=2p-1<0$.  In this case, the polynomial $\psi^\star(z)=(1-p)z-p$ has root $y=p/(1-p)<1$.  Then $\bm{H}_i=(\bm{1}_i^\T\;\;\bm{0}\;\;\cdots)$,\; $\bm{\Psi}=(y,y^2,\ldots)$,\; $\bm{\eta}_\infty^\T=(1-p\;\;\bm{0}\;\;\cdots)$, and $\smash{\sum_{\ell=1}^{k_p}p_\ell-\sum_{\ell=1}^{k_q}\ell q_\ell=2p-1}$.  Thus,
\begin{align*}
\det(\bm{\Omega}_i)=\begin{vmatrix}\sum_{j=1}^i y^j&-1\\ (1-p)y&2p-1\end{vmatrix}
=(2p-1)\sum_{j=1}^i y^j+(1-p)y
&=\frac{(2p-1)y(1-y^i)}{1-y}+(1-p)y\\
&=p\left(\frac{p}{1-p}\right)^i,
\end{align*}
and
\[
\sum_{j=0}^\infty\det(\bm{\Omega}_j)=p\left(1-\frac{p}{1-p}\right)^{-1}
=\frac{p(1-p)}{1-2p}<\infty.
\]
Then the stationary distribution is given by
\[
\pi_i=\frac{1-2p}{1-p}\left(\frac{p}{1-p}\right)^i=\left(1-\frac{p}{1-p}\right)\left(\frac{p}{1-p}\right)^i
\]
for each $i\in\{0,1,\ldots\}$, which agrees with the standard result \citep[e.g.,][]{feller1968,karlin1998}.
\end{ex}

\begin{ex}
Consider a one-sided reflecting random jump with $k_p=k_q=2$ and $\mu=2p_2+p_1-q_1-2q_2<0$.  In this case, the polynomial $\psi^\star(z)=p_2z^3+(p_1+p_2)z^2-(q_1+q_2)z-q_2$\\
\newpage
\noindent has three roots, exactly two of which have absolute value less than~$1$.  Let $y_1$ and $y_2$ denote these two roots.  Then
\bas\SADS
\bm{H}_i=\begin{pmatrix}1&\bm{0}&\bm{0}&\cdots\\0&\bm{1}_i^\T&\bm{0}&\cdots\end{pmatrix},\qquad
\bm{\Psi}=\begin{pmatrix}1&1\\y_1&y_2\\y_1^2&y_2^2\\ \vdots&\vdots\end{pmatrix},\qquad
\bm{\eta}_\infty=\begin{pmatrix}q_1+2q_2\\q_2\\ \bm{0}\\ \vdots\end{pmatrix},
\eas
and $\sum_{\ell=1}^{k_p}p_\ell-\sum_{\ell=1}^{k_q}\ell q_\ell=p_1+p_2-q_1-2q_2=\mu-p_2$.  Thus,
\begin{align*}
\det(\bm{\Omega}_i)
&=\begin{vmatrix}1&1&1\\ \sum_{j=1}^i y_1^j&\sum_{j=1}^i y_2^j&-1\\(q_1+2q_2)y_1+q_2y_1^2&(q_1+2q_2)y_2+q_2y_2^2&\mu-p_2\end{vmatrix}\\
&=\begin{vmatrix}1&1&1\\ (y_1-y_1^{i+1})/(1-y_1)&(y_2-y_2^{i+1})/(1-y_2)&-1\\(q_1+2q_2)y_1+q_2y_1^2&(q_1+2q_2)y_2+q_2y_2^2&\mu-p_2\end{vmatrix}\\
&=\begin{vmatrix}1-y_1&1-y_2&1\\ y_1-y_1^{i+1}&y_2-y_2^{i+1}&-1\\
-q_2y_1^3-(q_1+q_2)y_1^2+(q_1+2q_2)y_1&-q_2y_2^3-(q_1+q_2)y_2^2+(q_1+2q_2)y_2&\mu-p_2\end{vmatrix}\\
&\qquad\times\frac{1}{(1-y_1)(1-y_2)}\\
&=\begin{vmatrix}1-y_1&1-y_2&1\\ y_1-y_1^{i+1}&y_2-y_2^{i+1}&-1\\
-(p_1+p_2)y_1-p_2+(q_1+2q_2)y_1&-(p_1+p_2)y_1-p_2+(q_1+2q_2)y_2&\mu-p_2\end{vmatrix}\\
&\qquad\times\frac{1}{(1-y_1)(1-y_2)},
\end{align*}
noting that $q_2y_1^3+(q_1+q_2)y_1^2-(p_1+p_2)y_1-p_2=0$ and $q_2y_2^3+(q_1+q_2)y_2^2-(p_1+p_2)y_2-p_2=0$ since $y_1$ and $y_2$ are both roots of the recurrence relation.  Continuing, we have
\begin{align*}
\det(\bm{\Omega}_i)
&=\frac{1}{(1-y_1)(1-y_2)}
\begin{vmatrix}1-y_1&1-y_2&1\\ y_1-y_1^{i+1}&y_2-y_2^{i+1}&-1\\
p_2(y_1-1)-\mu y_1&p_2(y_2-1)-\mu y_2&\mu-p_2\end{vmatrix}\\
&=\frac{1}{(1-y_1)(1-y_2)}
\begin{vmatrix}1-y_1&1-y_2&1\\ 1-y_1^{i+1}&1-y_2^{i+1}&0\\
p_2(y_1-1)-\mu y_1&p_2(y_2-1)-\mu y_2&\mu-p_2\end{vmatrix}\\
&=\frac{1}{(1-y_1)(1-y_2)}
\begin{vmatrix}1-y_1&1-y_2&1\\ 1-y_1^{i+1}&1-y_2^{i+1}&0\\
-\mu y_1&-\mu y_2&\mu\end{vmatrix}\\
&=\frac{\mu}{(1-y_1)(1-y_2)}
\begin{vmatrix}1&1&0\\ 1-y_1^{i+1}&1-y_2^{i+1}&0\\
-y_1&-y_2&1\end{vmatrix}\\
&=\frac{\mu\left[(1-y_2^{i+1})-(1-y_1^{i+1})\right]}{(1-y_1)(1-y_2)}
=\frac{\mu\left(y_1^{i+1}-y_2^{i+1}\right)}{(1-y_1)(1-y_2)}.
\end{align*}
Then
\begin{align*}\SADS
\sum_{j=0}^\infty\det(\bm{\Omega}_j)
&=\frac{\mu}{(1-y_1)(1-y_2)}\sum_{j=0}^\infty\left(y_1^{j+1}-y_2^{j+1}\right)\\
&=\frac{\mu}{(1-y_1)(1-y_2)}\left(\frac{y_1}{1-y_1}-\frac{y_2}{1-y_2}\right)
=\frac{\mu(y_1-y_2)}{(1-y_1)^2(1-y_2)^2}
<\infty,
\end{align*}
and thus the stationary distribution is given by
\[
\pi_i=\frac{(1-y_1)(1-y_2)}{y_1-y_2}\left(y_1^{i+1}-y_2^{i+1}\right)
\]
for each $i\in\{0,1,\ldots\}$.
\end{ex}
\section{%
An Application 
of Simple Random Leap Theory}
\label{sec:example}

As an example, we consider the casino game of roulette.  In standard American roulette, bets are placed and settled according to the eventual resting place of a ball spun around a wheel with $38$ spaces (numbered $1$--$36$, with additional spaces marked $0$ and $00$).  A wide variety of bets are possible, but for our purposes we consider only the so-called column bet, in which the gambler essentially bets that the result of the spin will be a positive number congruent to $0$, $1$, or $2$ modulo~$3$.  (Note, however, that a bet on $\{3,6,9,\ldots,36\}$ does \emph{not} win on $0$ or~$00$.)  A column bet pays at $2$-to-$1$, i.e. a gambler who bets \$1 will either lose this \$1 or win \$2.  However, consider a gambler who sometimes bets \$1 on a single column and sometimes bets \$1 on each of two columns simultaneously, making each decision with probability $1/2$.  Then on each spin of the wheel, the gambler's winnings in dollars (where negative numbers indicate losses) will be either $-2$, $-1$, $+1$, or $+2$, and these results will occur with probabilities $7/38$, $13/38$, $12/38$, and $6/38$, respectively.  Suppose this gambler enters the casino with $i$ dollars and will leave upon either going broke or reaching at least $N$ dollars.  Also, to maintain the structure of the absorbing simple random leap, suppose that if the gambler has \$1 remaining but wishes to place two bets, he or she will borrow \$1 from a friend, which will be repaid immediately after a win but not after a loss (since after a loss, the gambler will have no money with which to repay the debt).
Then the results of Theorems~\ref{thm:u}~and~\ref{thm:v} may be applied to find the probability that the gambler reaches the target~$N$, as well as the expected number of spins until the gambler reaches the target or leaves the casino.  These quantities are shown in Tables~\ref{tab:u}~and~\ref{tab:v} for each $N\in\{5,10,15,20,25\}$ and each $i\in\{0,\ldots,N\}$.

\section{Conclusion}
\label{sec:concl}

In this paper, we have extended some basic properties of the
classical
simple random walk to a simple random leap with maximum step size $\pm k$ in each direction.  Specifically, we have used the theory of linear recurrence relations to write the absorption probabilities~$u_i$ and expected absorption times~$v_i$ as determinants of matrices involving powers of the roots of certain characteristic polynomials.  Furthermore, we provided elementary proofs of the necessary and sufficient condition for recurrence of the simple random leap on~$\mathbb{Z}$, as well as of the nonexistence of a stationary distribution.

\section*{Acknowledgments}
\label{sec:ack}

The authors wish to thank Apoorva Khare for his insight and generous assistance.
This work was partially funded by the US National Science Foundation under grants
DMS-CMG-1025465, AGS-1003823, DMS-1106642, DMS-CAREER-1352656,
and the US Air Force Office
of Scientific Research grant award FA9550-13-1-0043.

\bibliographystyle{ims}
\bibliography{references}

\appendix
\section{Appendix}

An alternative proof of Theorem~\ref{thm:pi-limit} relies upon a technique that is in fact more general.  Specifically, consider a sequence of irreducible Markov chains $\{\{\smash{X_n^{(N)}}:n\ge0\}:N\ge N_0\}$ such that the state space of $\{\smash{X_n^{(N)}}:n\ge0\}$ is $\{0,\ldots,N\}$ for each $N\ge N_0$.  Let $\bm{P}^{(N)}$ denote the associated transition matrices, and suppose that there exists a positive integer~$k$ such that $\smash{P^{(N)}_{ij}=0}$ for all~$N$ if $|i-j|>k$.  Let $\bm{\pi}^{(N)}$ refer to the stationary distribution associated with each transition matrix~$\bm{P}^{(N)}$.  Then, let $\tilde{\bm{\pi}}^{(N)}$ and $\tilde{\bm{P}}^{(N)}$ denote the natural extensions of $\bm{\pi}^{(N)}$ and $\bm{P}^{(N)}$ to the state space $\{0,1,\ldots\}$, i.e.,
\[
\tilde{\bm{\pi}}^{(N)}=\begin{pmatrix}\bm{\pi}^{(N)}\\0\\0\\ \vdots\end{pmatrix},
\qquad
\tilde{\bm{P}}^{(N)}=\begin{pmatrix}
\bm{P}^{(N)}&\bm{0}&\bm{0}&\cdots\\
\bm{0}&1&0&\cdots\\
\bm{0}&0&1&\cdots\\
\vdots&\vdots&\vdots&\ddots\end{pmatrix}.
\]
Finally, let $\bm{P}$ be a matrix such that for each $i\in\{0,1,\ldots\}$,
$P_{ij}=\smash{\tilde{P}_{ij}^{(N)}}$ for all~$j$ if $N\ge i+k$.
Then we have the following result.

\begin{thm}\label{thm:sd-extend}
If $\bm{\pi}=(\pi_0\;\;\pi_1\;\;\cdots)^\T$ is a distribution such that
$\smash{\tilde{\pi}_i^{(N)}}\to\pi_i$ uniformly in $i$ as $N\to\infty$, then $\bm{\pi}\bm{P}=\bm{\pi}$.
\end{thm}

\begin{proof}
Let $\varepsilon>0$.  Since $\bm{\pi}$ is a distribution, $\lim_{i\to\infty}\pi_i=0$.  Then by this fact and the condition of the theorem, there exists $M$ large enough so that both $|\smash{\tilde{\pi}_i^{(M)}}-\pi_i|<\varepsilon/(4k+4)$ for every $i\ge0$ and
$\smash{\tilde{\pi}_i^{(M)}}<\varepsilon/(4k+2)$ for every $i>M-k$.  Then for every $j\ge0$,
\begin{align*}
\left|\pi_j-\sum_{i=0}^\infty\pi_iP_{ij}\right|
&=\left|\pi_j-\smashoperator{\sum_{i=\max\{0,j-k\}}^{j+k}}\pi_iP_{ij}\right|\\
&\le\left|\pi_j-\smashoperator{\sum_{i=\max\{0,j-k\}}^{j+k}}\tilde{\pi}_i^{(M)}\tilde{P}_{ij}^{(M)}\right|
+\left|\sum_{i=\mathrlap{\max\{0,j-k\}}\;\;}^{j+k}\tilde{\pi}_i^{(M)}\left(\tilde{P}_{ij}^{(M)}-P_{ij}\right)\right|\\
&\qquad+\left|\sum_{i=\mathrlap{\max\{0,j-k\}}\;\;}^{j+k}\left(\tilde{\pi}_i^{(M)}-\pi_i\right)P_{ij}\right|\\
&
\le\left|\pi_j-\tilde{\pi}_j^{(M)}\right|
+(2k+1)\sup_{i>M-k}\tilde{\pi}_i^{(M)}+(2k+1)\sup_{i\ge0}\left|\tilde{\pi}_i^{(M)}-\pi_i\right|\\
&\le(2k+1)\frac{\varepsilon}{4k+2}+(2k+2)\frac{\varepsilon}{4k+4}=\varepsilon,
\end{align*}
noting that $\smash{\tilde{P}^{(M)}_{ij}}\ne P_{ij}$ only if $i>M-k$.
\end{proof}

This theorem may be applied to the reflecting random leap as follows.  Denote the vector~$\bm{\pi}$ and the matrices~$\bm{W}_i$ in Theorem~\ref{thm:pi} as $\bm{\pi}^{(N)}$ and $\bm{W}_i^{(N)}$ to explicitly indicate their dependence on~$N$.  Then let $\tilde{\bm{\pi}}^{(N)}$ denote the natural extension of $\bm{\pi}^{(N)}$ to the state space~$\{0,1,\ldots\}$.

\begin{cor}
If $\bm{\pi}=(\pi_0,\pi_1,\ldots)$ is a distribution such that
\[
\tilde{\pi}_i^{(N)}=\frac{\det\!\left[\bm{W}_i^{(N)}\right]}{\sum_{j=0}^N\det\!\left[\bm{W}_j^{(N)}\right]}\to\pi_i
\]
uniformly in $i$ as $N\to\infty$, then $\bm{\pi}$ is the stationary distribution of the one-sided reflecting random leap.
\end{cor}

\begin{proof}
This result is a direct application of Theorem~\ref{thm:sd-extend}.
\end{proof}

Thus, if the stationary distribution for the two-sided reflecting random leap on $\{0,\ldots,N\}$ in Theorem~\ref{thm:pi} can be shown to converge uniformly to some distribution $\bm{\pi}$ on $\{0,1,\ldots\}$, then this limiting distribution is the stationary distribution of the corresponding one-sided reflecting random leap.


\begin{table}[p]
\begin{center}
\begin{tabular}{c@{\hspace{3em}}rrrrr}\toprule
\multirow{2}{*}{$i$}&\multicolumn{5}{c}{\phantom{.0}$N$\hspace*{6pt}}\\\cmidrule{2-6}
&5\phantom0&10\phantom0&15\phantom0&20\phantom0&25\phantom0\\ \midrule
1&.1978&.0829&.0447&.0266&.0166\\
2&.3541&.1490&.0803&.0477&.0269\\
3&.5445&.2272&.1223&.0727&.0455\\
4&.7252&.3098&.1668&.0992&.0620\\
5&1.0000&.3996&.2152&.1280&.0800\\\midrule[0.2\lightrulewidth]
6&&.4968&.2674&.1590&.0995\\
7&&.6010&.3239&.1926&.1205\\
8&&.7170&.3850&.2289&.1432\\
9&&.8297&.4510&.2682&.1677\\
10&&1.0000&.5224&.3106&.1943\\\midrule[0.2\lightrulewidth]
11&&&.5997&.3565&.2230\\
12&&&.6826&.4061&.2540\\
13&&&.7749&.4598&.2876\\
14&&&.8645&.5178&.3238\\
15&&&1.0000&.5805&.3631\\\midrule[0.2\lightrulewidth]
16&&&&.6484&.4055\\
17&&&&.7212&.4513\\
18&&&&.8023&.5009\\
19&&&&.8810&.5545\\
20&&&&1.0000&.6124\\\midrule[0.2\lightrulewidth]
21&&&&&.6751\\
22&&&&&.7424\\
23&&&&&.8173\\
24&&&&&.8900\\
25&&&&&1.0000\\\bottomrule
\end{tabular}
\end{center}
\caption{Probability~$u_i$ that the gambler
described in
Section~\ref{sec:example}
with
an initial stake of $i$~dollars reaches at least $N$ dollars before going broke.}
\label{tab:u}
\end{table}


\begin{table}[p]
\begin{center}
\begin{tabular}{c@{\hspace{3em}}rrrrr}\toprule
\multirow{2}{*}{$i$}&\multicolumn{5}{c}{\phantom{.0}$N$\hspace*{6pt}}\\\cmidrule{2-6}
&5\phantom{00}&10\phantom{00}&15\phantom{00}&20\phantom{00}&25\phantom{00}\\ \midrule
1&2.6764&5.3096&7.5377&9.3944&10.9165\\
2&3.5075&8.2208&12.2248&15.5612&18.2963\\
3&3.6019&10.8523&16.9560&22.0422&26.2118\\
4&\phantom{00}2.8784&\phantom{0}12.5231&\phantom{0}20.8465&\phantom{0}27.7819&\phantom{0}33.4674\\
5&.0000&13.3454&24.0806&33.0273&40.3618\\\midrule[0.2\lightrulewidth]
6&&13.1845&26.5333&37.6511&46.7655\\
7&&12.0244&28.1613&41.6276&52.6672\\
8&&9.5941&28.8924&44.8972&58.0180\\
9&&6.4972&28.6540&47.4045&62.7755\\
10&&.0000&27.3721&49.0870&66.8909\\\midrule[0.2\lightrulewidth]
11&&&24.9425&49.8782&70.3123\\
12&&&21.3471&49.7053&72.9833\\
13&&&16.2060&48.4906&74.8429\\
14&&&10.4765&46.1478&75.8255\\
15&&&.0000&42.5925&75.8596\\\midrule[0.2\lightrulewidth]
16&&&&37.6988&74.8685\\
17&&&&31.4618&72.7687\\
18&&&&23.3795&69.4709\\
19&&&&14.7939&64.8754\\
20&&&&.0000&58.8861\\\midrule[0.2\lightrulewidth]
21&&&&&51.3547\\
22&&&&&42.2899\\
23&&&&&31.0589\\
24&&&&&19.4157\\
25&&&&&.0000\\\bottomrule
\end{tabular}
\end{center}
\caption{Expected number of spins~$v_i$ on which the gambler
described in
Section~\ref{sec:example}
with
an initial stake of $i$~dollars reaches at least $N$~dollars or goes broke.}
\label{tab:v}
\end{table}

\end{document}